\documentclass[11pt]{article}
\usepackage[utf8]{inputenc}
\usepackage[english]{babel}
\usepackage{layout, graphicx, amsmath, amsthm, amsfonts, amssymb, euscript, enumerate, 
bbm, graphicx, fancyhdr,xcolor,stmaryrd,mathrsfs,pdfsync}

\usepackage{psfrag}
\usepackage{hyperref}

\makeatletter
\def\th@newremark{\th@remark\thm@headfont{\bfseries}}
\makeatletter

\theoremstyle{newremark}

\newcommand{\Co}{\mathcal{C}}

\def\N{\mathbb{N}}
\def\P{\mathbf{P}}
\def\R{\mathbb{R}}
\def\E{\mathbf{E}}

\def\ind{\mathbf{1}}

\def\m{\mathfrak{m}}

\definecolor{coultitre2}{rgb}{0.41,0.05,0.05}  
\definecolor{coultitre}{RGB}{88,189,63}
\definecolor{fondtitre}{rgb}{0.20,0.43,0.09}
\definecolor{lightblue}{cmyk}{0.83,0.24,0,0.12}
\definecolor{blueCollege}{RGB}{85, 122, 154}
\definecolor{redIDF}{RGB}{243, 59, 40}
\definecolor{purpleRDM}{RGB}{142, 98, 168}
\definecolor{lightpurple}{RGB}{204, 153, 255} 
\definecolor{blueInkscape}{RGB}{0,0,255}
\definecolor{newred}{RGB}{201,20,35}
\definecolor{greenlantern}{RGB}{109,163,49}
\definecolor{verdesi}{RGB}{0,153,0}
\definecolor{amarilloEi}{RGB}{255, 204, 102}
\definecolor{gray2}{RGB}{192,192,192}

\def\I{\mathbb{I}}

\def\M{\mathbb{M}}

\def\LL{\mathcal{L}}

\newtheorem{theorem}{Theorem}[section]
\newtheorem{corollary}[theorem]{Corollary}
\newtheorem{definition}[theorem]{Definition}
\newtheorem{lemma}[theorem]{Lemma}
\newtheorem{proposition}[theorem]{Proposition}
\newtheorem{remark}[theorem]{Remark}

\title{Evolving genealogies for branching populations \\ under selection and competition}
\author{Airam Blancas\thanks{Department of Statistics, Instituto Tecnol\'ogico Aut\'onomo de M\'exico, M\'exico; E-mail: \texttt{airam.blancas@itam.mx}} , Stephan Gufler\thanks{Goethe Universit\"at, Institut f\"ur Mathematik, 60629 Frankfurt am Main, Germany; E-mail: \texttt{gufler@math.uni-frankfurt.de}, \texttt{wakolbinger@math.uni-frankfurt.de},} , Sandra Kliem\thanks{Universit\"at Leipzig, Mathematisches Institut, Augustusplatz 10, 04109 Leipzig; E-mail: \texttt{kliem@math.uni-leipzig.de}}, \\ Viet Chi Tran\thanks{LAMA, Univ Gustave Eiffel, Univ Paris Est Creteil, CNRS, F-77454 Marne-la-Vall\'ee, France; E-mail: \texttt{chi.tran@univ-eiffel.fr}} \, and Anton Wakolbinger \!\!$^\dagger$
}

\begin{document}
\maketitle
\vspace{-.2in}

\begin{abstract}
For a continuous state branching process with two types of individuals which are subject to selection and density dependent competition, we characterize the joint evolution of population size, type configurations and genealogies as the unique strong solution of a system of SDE's.
 Our \mbox{construction} is achieved in the lookdown framework and provides a synthesis as well as a generalization of cases considered separately in two seminal papers by Donnelly and Kurtz (1999), namely fluctuating population sizes under neutrality, and selection with constant population size. As a conceptual core in our approach we introduce the  {\em selective lookdown space} which is obtained  from its neutral counter\-part through a state-dependent thinning of  ``potential'' selection/competition events whose rates interact with the evolution of the type densities. The updates of the genealogical distance matrix at the \mbox{``active''} selection/competition events are obtained through an appropriate sampling from the \mbox{selective} lookdown space. The solution of the above mentioned system of SDE's is then mapped into the joint evolution  of  population size and symmetrized type configurations and genealogies, i.e. marked distance matrix distributions.   By means of Kurtz' Markov mapping theorem, we characterize the latter process as the unique solution of a  martingale problem. For the sake of transparency we restrict the main part of our presentation to a prototypical example with two types, which contains  the essential features. In the final section we outline an extension  to processes with multiple types including mutation.
%
\end{abstract}



{\small
\tableofcontents
}

\section{Introduction}  \label{intro-section}
The aim of our paper is to give a pathwise construction for the joint evolution of population size, type frequencies and genealogies in  a continuous state branching process with interactions due to type dependent selective advantage in reproduction and type density dependent competition. Such processes model large populations whose individuals are distinguished by their types. The sizes of the populations and their type structures are fluctuating due to individual births and deaths, where certain types may have a selective advantage in the fecundity, and others may have a disadvantage against some other types, say in the competition for resources. We are interested here in these dynamics but also in that of the genealogies of the individuals composing these populations, which consist of the collection of their ancestral paths, i.e. the succession of their ancestors with their types.  We demonstrate the strength of the approach in a prototypical example with two types, one of them
having a selective advantage, the other one having a competitive disadvantage.
This restriction is mainly for presentational reasons; in Section \ref{sec:multitype} we will outline an extension
to more general processes with multiple types, including mutations.

We take the so-called lookdown approach that has been developed by Donnelly and Kurtz in order to construct and study the evolution of continuum populations with a general type space in terms of a countably infinite particle system. In two seminal papers, these authors treated two distinct cases: that of populations with constant sizes under selection (and recombination) \cite{MR1728556}
and that  of neutral populations with fluctuating population sizes \cite{MR1681126}. In the present work we consider selection and competition combined with fluctuating sizes. One of our key results, Theorem~\ref{Gprop},  extends ideas in the proof of \cite[Theorem~4.1]{MR1728556} to a  situation where the total mass is a stochastic process whose dynamics depends on the type frequencies,  and thus opens the way for a synthesis of the settings of \cite{MR1728556} and~\cite{MR1681126}.
In both of these papers  the evolution of the (relative) type frequencies and the genealogies are encoded in an infinite particle system that describes the reproductive events. In~\cite{MR1728556} the population size (or total mass of the continuum population) is assumed to be constant, while in \cite{MR1681126} it is accounted for in a separate process, which is autonomous due to the neutral setting considered in that paper.
This is no longer the case in our setting where additional births and deaths occur in the infinite particle system due to selection and competition which depend on and also impact the evolution of the population size.

While many considerations pertaining to genealogies and ancestral lineages are already present in and between the lines of \cite{MR1728556} and \cite{MR1681126}, the power of the lookdown approach for studying evolving genealogies has unfolded more recently, several years after Evans \cite{Evans2000} characterized Kingman's coalescent as a random metric space. The lookdown representation of the evolving populations in terms of exchangeable particle systems comes with a graphical representation that   provides a genealogy in a natural way.  A~central tool for proving Theorem~\ref{Gprop} are the {\em sampling measures} on the {\em neutral lookdown space}, which is the completion of $\mathbb R_+ \times \mathbb N$ with respect to the (random) semi-metric given by the (neutral) genealogical distances.
The concept of the (neutral) lookdown space has recently been introduced in \cite{Gu1} to obtain (in the neutral case and for constant population size) a pathwise construction of  tree-valued Fleming-Viot processes. This is reviewed in Section \ref{neut} and Theorem \ref{Gprop} is proved in Section \ref{sec_itscheme}. In  Section \ref{genealogy} (see also the preview in Sec.~\ref {sec:pathwiseconstructionofgenealogies}) we will construct what we call the {\em selective lookdown space with fluctuating population size}. This space will carry the sampling measures which will serve to update the type configuration as well as the genalogical distances at selective events, see Sec.~\ref{updateR}.
The selective lookdown space provides a `global' description of the genealogies as a random metric space. Our paper provides also a `dynamical' construction of the latter, with the Theorem~\ref{zetaRG} that establishes a system of SDE's (with unique strong solution) for the joint process of total mass, genealogical distance matrix, and type configuration. Exploiting the exchangeability that comes with the concept of sampling, we then turn to the {\em symmetrization} or ``unlabelling'' of the lookdown genealogies. As states that describe the type distributions and genealogies, we use here the isomorphy classes of marked ultrametric measure spaces
introduced in~\cite{grevenpfaffelhuberwinter2009,DGP11,MR3024977} which can be thought of as marked distance matrix distributions. Theorem~\ref{martprobth} then characterizes the joint evolution of genealogies and population size in terms of a well-posed martingale problem. This result is proved by a two-fold application of Kurtz' Markov mapping theorem, see Sections~\ref{twomp} and \ref{symmgen}.

 While we provide a  solution of this martingale problem from specified sources of randomness (the Brownian motion $\mathcal W$ and Poisson point processes $\mathcal L$ and $\mathcal K$ defined in Section \ref{secmain}), a more common approach for showing the existence of a solution would be to deal with tightness of finite approximations. In our setting, apart from being less constructive, this approach may cause serious technical problems. One of the few papers in which a martingale problem for  continuum tree-valued processes including pairwise (competitive) individual interactions (and fluctuating total mass) has been treated via tightness is~\cite{KliemWinter2018}. Note however that there the  distances between individuals are measured in terms of numbers of mutations, whereas we measure distances in terms of times back to the most recent common ancestor of the two individuals.
In \cite{MR3024977} tree-valued Fleming-Viot processes with mutation and type-frequency dependent selection are constructed, but there constant population size is assumed.
Our contribution here is to also provide a lookdown representation of the genealogies that completes the picture: we construct the genealogies as a metric space, endowed with sampling measures, which, in an appropriate sense, locally look like the neutral genealogies in absence of selection and competition, with modifications related to the selective and competitive events.


We use the Poisson process $\mathcal L$ of {\em lookdown events} (see Section \ref{secmain}) to encode the elements of  neutral genealogies, but note that there exist also alternative routes for doing this. One of them is along the {\em continuum random tree} and Brownian excursions (\cite{Aldous1991,BeBe2009,LeGall2005} or \cite[Ch.~4]{GuDis}), with certain deformations of these objects to model competition (\cite{BeFiFo2018,LePaWa2015,PaWa2011}) although the introduction of types is not straightforward in these models and in the cited references
the competition depends on the individuals' left-right order encoded in the excursion.
Another one is Kurtz and Rodrigues' {\em lookdown representation with a continuum of levels} \cite{KuRo2011}, which has recently been extended by Etheridge and Kurtz \cite{EtheridgeKurtz} to a variety of models including selection and competition, but with less emphasis on evolving genealogies.

Recent work on evolving genealogies in the neutral case, with a focus on heavy-tailed offspring distributions, has been reviewed in \cite{KeWa20}.
Evolving ancestral path configurations under competition are studied in \cite{meleardtran_suphist,kliem} or \cite{calvezhenrymeleardtran,henrymeleardtran}, building on the framework of {\em historical processes} which was pioneered in \cite{dawsonperkinsAMS,dynkin91}.
Inference methods in the presence of selection, varying population size and evolving population structure are described in \cite{lepersportebilliardmeleardtran}, extending results of \cite{billiardferrieremeleardtran}; in the latter models, a time-scale separation allows to treat separately the type structure and population size on the one hand, and the genealogies on the other hand. In the present paper, however, we deal simultaneously with interactions, \nopagebreak demography and genealogies.

\section{Model and main results}\label{secmain}

\subsection{Population size and type frequencies}

In order to make the conceptual novelties and essentials as transparent as possible, we will restrict ourselves in the main part of this work to a population with only two types, $A$ and $B$. An extension of the results  to more general type spaces and including mutations is outlined in Section \ref{sec:multitype}.
We will denote by $\mathbb I = \{A, B\}$  the type space, and by $\xi^A$ and $\xi^B$ the processes in continuous time corresponding to the sizes of the type $A$- and type $B$-populations. Intuitively, these populations consist of a continuum of individuals with infinitesimal masses; the concept of  {\em sampling measures}, which we will recover also in the {\em selective lookdown space}, makes this intuition rigorous.

The \textit{population size} or total \textit{population mass} at a time $t>0$ is $\xi_t=\xi_t^A+\xi^B_t$. We define   the \textit{type frequencies} or proportions of types $A$ and $B$ as
\[\mu_t^A= \frac{\xi^A_t}{\xi_t}\ind_{ \{\xi_t>0\} }\, , \qquad \mu_t^B= \frac{\xi^B_t}{\xi_t}\ind_{ \{\xi_t>0\} }.\]
The system we are going to consider as a prototypical case is a two-type Feller branching diffusion with interactions
\begin{align}
\begin{split}
\label{twotype}
d\xi_t^A&= b \xi_t^A dt -c\xi_t^A \xi^B_t\, dt+ \sqrt{\xi_t^A} dW_t^{A} \\
d\xi_t^B&= -c \xi_t^B  \xi_t^A \,dt + \sqrt{\xi_t^B} dW_t^{B}, \qquad \qquad b\ge 0,\, c \ge 0, 
\end{split}
\end{align}
where $W^A$ and $W^B$ are independent standard Brownian motions. The processes $W^A$ and $W^B$ drive the fluctuations due to natural births and deaths in the diffusion limit of branching populations.  The nonnegative constant  $b$ is the coefficient of the intensity of additional births of type $A$-individuals due to their enhanced fecundity, whereas $c$ is the intensity of additional deaths  of individuals due to their competition against all individuals of the opposite type. Such a system of equations can be seen as arising from the limit of finite particle systems. In \cite[Ch.~9 Sec.~2 p.~392]{ethierkurtz} this is proved in the case $c=0$; then both $\xi^A$ and $\xi^B$ are independent Feller diffusions.

The following proposition, whose proof we will include  at the end of Section~\ref{sec_sel-gen},  guarantees existence and uniqueness of a strong solution of \eqref{twotype} and states the long time behavior of  $(\xi^A,\xi^B)$, namely that the process $\xi^B$ gets extinct in finite time almost surely, while $\xi^A$ becomes either trapped in $0$ or diverges to~$+\infty$.
\begin{proposition}\label{strongsol_type_process} Let $\xi^A_0$ and $\xi^B_0 $ be strictly positive and let $W^A$ and $W^B$ be independent standard Brownian motions. Then there exists a unique strong solution $(\xi^A,\xi^B)$ to \eqref{twotype} for all times $t\in\R_+$.  
Moreover, for the extinction times of $\xi^h$ for $h\in\{A,B\}$, which are defined by $\tau^h_0:=\inf\{t\ge 0:\: \xi^h_t=0\}$, we have:
\begin{enumerate}
\item[(i)]$\P\left(\tau^B_0<\infty\right)=1$.
\item[(ii)] If $b>0$,
\begin{equation*}
	\P\left(\tau^A_0<\infty\right)\in(0,1),\quad
\P\left(\tau^A_0<\infty\middle|\xi^{A}_{\tau^B_0}>0\right)\in(0,1),\quad
\P\left(\lim_{t\to\infty}\xi^A_t=\infty\middle|\tau^A_0=\infty\right)=1
\end{equation*}
\item[(iii)] If $b=0$, then $\P\left(\tau^A_0<\infty\right)=1$.
\end{enumerate}
\end{proposition}
As announced in the Introduction, our main goal is the characterization of {\em evolving marked genealogies} that underlie the system \eqref{twotype}. It turns out that an accessible way to this goal leads via the {\em total mass process} $\xi = \xi^A+\xi^B$.
Adding the two equations in \eqref{twotype} leads to the following stochastic differential equation (SDE) for $\xi$ and the type proportions $\mu_t^A$, $\mu_t^B$:
\begin{equation}\label{totmass}
d\xi_t= (b \mu_t^A \xi_t - 2c  \mu_t^A \mu_t^B \xi_t^2) dt + \sqrt{\xi_t}\,  d {W}_t, \\
\end{equation}
with $W$ a standard Brownian motion. Equation \eqref{totmass} cannot be solved in terms of $W$ without knowing the type frequencies $\mu_t^A, \mu_t^B$ (which in turn are obtained from \eqref{twotype}); in this sense \eqref{totmass}  is not autonomous. In addition to~$W$, which takes care of the fluctuations of the population size in the interplay with the current type proportions, the other drivers of the evolving marked genealogy that will trigger the (neutral and selective) reproductive events will be Poisson point processes that come up in the lookdown framework described in Section  \ref{section:LDrepresentation}. Closing the circle, Proposition \ref{prop:projection} will then guarantee that the pair $(\xi^A, \xi^B)$ can be restored from the total mass process $\xi$ together with the evolving marked genealogy $Y$, thus rendering a weak solution of \eqref{twotype}. The process $Y$ will be described in Section \ref{SecGen} and defined in Section~\ref{mppopsizeandgenealogy}.

The following time change will be instrumental (see also  \cite{MR1681126}):
 \begin{equation}\label{tch}
  t \mapsto s=s(t) := \int_0^t \frac 1{\xi_v}\,  dv,   \qquad \zeta_{s(t)} = \xi_t\, .\end{equation}
For reasons explained in Section \ref{section:LDrepresentation}, the timescale $s$ will be called the {\em lookdown timescale}. In Section \ref{section:LDrepresentation}, we will provide a system of stochastic differential equations that describes  in this timescale a population size process $(\zeta_s)$ together with an evolving type configuration $(G_s)$ whose state space is $\mathbb I^{\mathbb N}$  (with $\mathbb I = \{A,B\}$) and to which we will associate a process of type frequencies $(\mu^{G_s})$, see Theorem \ref{Gprop}.
 As a corollary, transforming back to the timescale $t$ via \eqref{tch}, the resulting process $(\zeta_{s(t)}\mu^{G_{s(t)}}\{A\}, \zeta_{s(t)}\mu^{G_{s(t)}}\{B\})$ will provide a  weak  solution of \eqref{twotype}, see Proposition \ref{prop:projection}.

In the neutral case ($b=c=0$), $(\xi_t)$ is a standard Feller diffusion, and the  process $(\mu_t^A)$ after the time change \eqref{tch} turns into a standard Wright-Fisher diffusion (e.g. \cite[Ch. IV.8]{ikedawatanabe}). The correspondence $(\xi_t^A, \xi_t^B) \leftrightarrow ((\xi_t), (\mu_t^A, \mu_t^B))$ is thus an interactive counterpart of  Perkins' desintegration of super-Brownian motion into a Feller branching diffusion and a time-changed Fleming-Viot process (see \cite{Eth2000} p. 83, \cite{Pe1992}).

\subsection{Genealogies}\label{SecGen}
The {\em marked genealogy} of the continuum population at some fixed time is described by the joint distribution of pairwise genealogical distances and types of a sequence of individuals that is drawn    i.i.d. according to a prescribed sampling measure. In order to formalize this, and to define the  space  of marked genealogies, we recall a few concepts. In our context, genealogical distances of contemporaneous individuals are described by a semi-ultrametric, i.e.~a semi-metric $d$ that satisfies the strong triangle inequality $\max\{d(x,y), d(y,z)\} \ge d(x,z)$. The prefix {\em semi} means that $d(x,y) = 0$ does not imply $x=y$, corresponding to the fact that at the time of a reproduction event, the ``mother'' and her ``daughter'' have genealogical distance $0$, while being considered as different individuals.

Marked metric measure spaces have been introduced by Depperschmidt, Greven, and Pfaffelhuber~\cite{DGP11}.
An {\em $\mathbb I$-marked ultrametric measure} space is a triple $(\tau, d, \mathrm m)$ where $(\tau, d)$ is a complete, separable ultrametric space and  $\mathrm m$ is a probability measure on the Borel sigma algebra on the product space $\tau \times \mathbb I$. In our context, such spaces $(\tau, d)$ will arise  as completions of semi-ultrametric spaces, after first identifying elements of distance zero, see  Definition \ref{defproper} a) in Sec. \ref{sec:pathwiseconstructionofgenealogies}.

The marked distance matrix distribution of an $\mathbb I$-marked ultrametric measure space $(\tau, d , \mathrm m)$ is defined as the distribution of $((d(V_i, V_j))_{i,j\in \mathbb N}, (H_i)_{i\in \mathbb N} )$ where $(V_i, H_i)_{i\in \mathbb N}$ is a sequence in $\tau \times \mathbb I$, i.i.d.~with distribution $\mathrm m$. (Here and below, $\mathbb N=\{1,2,\ldots\}$ denotes the set of natural numbers. Recall also that $\mathbb I=\{A,B \}$). Marked ultrametric measure spaces with the same marked distance matrix distribution are called {\em isomorphic}.

The space of isomorphy classes of $\mathbb I$-marked ultrametric measure spaces will be denoted by~$\mathbb M$, and will be called the {\em space of marked genealogies}. This space $\M$, equipped with the marked Gromov-weak topology in which elements of $\M$ converge if and only if the associated marked distance matrix distributions weakly converge, is Polish \cite{DGP11}. In Theorem \ref{martprobth} we will characterize an $(0,\infty) \times \mathbb M$-valued process $(\xi, Y)$ by a stopped martingale problem (in the sense of \cite[Ch. 4.6]{ethierkurtz}). The first component of this process will describe the population size, and will give a weak solution of \eqref{totmass}. The second component will describe the marked genealogy, with the type frequencies being a measurable function of the latter.

\subsection{Lookdown representation of the joint process of population size, type frequencies, and genealogies}\label{section:LDrepresentation}

We are going to provide a representation of the just mentioned process $(\xi, Y)$ in terms of a process $(\zeta, R, G)$ which will be the unique strong solution of a system of SDE's in the time scale~\eqref{tch}, see Theorems \ref{Gprop} and~\ref{zetaRG}. The process $R$ will take its values in the semi-ultrametrics on $\mathbb N$ (which we will address as {\em distance matrices} for short). The underlying graphical representation includes, in addition to a Brownian motion $\mathcal W$,  a pair $(\mathcal L, \mathcal K)$ of Poisson point processes (defined in Sec. \ref{sec:pathwiseconstructionLD}). The triple $(\mathcal W, \mathcal L, \mathcal K)$ does not only drive the process $(\zeta,G)$ in terms of an SDE (see Theorem \ref{Gprop}), but also the process $R$, see~\eqref{eqRG}.

We will deduce in Proposition \ref{MPA} that $(\zeta,R,G)$ solves a well-posed martingale problem. This will be an essential ingredient for the proof of Theorem \ref{martprobth}, which provides the characterization of $(\xi, Y)$ in terms of a well-posed martingale problem. This proof, like the one of Proposition \ref{MPA}, will rely on an application of Kurtz' Markov mapping theorem \cite{Kurtz98}, which for our purposes turns out to be more adequate than its modification in \cite{EtheridgeKurtz} (see Remark \ref{KMMTH}).

Individuals living in the  lookdown system at time $s$ are coded by $(s,i)$, $i=1,2,\ldots$. (As we will see from the constructions explained in Sec. \ref{sec:pathwiseconstructionofgenealogies}, this is only a subset of the uncountably many individuals living at time $s$, namely the subset consisting of those individuals who have an offspring that survives for some positive amount of time.) The second component of $(s,i)$ is called \textit{level}; it labels the individuals alive at time~$s$ and having an offspring at some time strictly larger than~$s$. The graphical construction will allow to reconstitute the ancestral paths of the individuals $(s,i)$. The evolution of the genealogical distances and the types  of these individuals will be described by the process
\[X=(R,G) = ((R_s,G_s)_{s\in[0,\sigma)}),\]
where $\sigma$ is the time at which $\zeta$ goes to extinction or explodes,
\begin{equation}\label{e:def-sigma}
\sigma:= \inf\{s\ge 0: \zeta_s = 0 \mbox{ or }  \zeta_{s-} = \infty\}.
\end{equation}
The second component $G_s = (G_s(i))_{i\in \mathbb  N} \in \mathbb I^{\mathbb N}$  of this process is the type configuration at time $s$.
The first component  $R_s= (R_s(i,j))_{{i,j} \in \mathbb N}$ is a random semi-ultrametric on~$\mathbb N$ that describes the genealogical distances between the individuals at time $s$ in the time scale of the interactive branching system \eqref{twotype}. That is, if the most recent common ancestor of $(s,i)$ and $(s,j)$ lived at time $s'<s$, then $R_s(i,j)=2(t(s)-t(s'))$, with
\begin{align}\label{inversetch}
t(s) := \int_0^s \zeta_u\, du
\end{align}
 being the inverse of the time change \eqref{tch}.

We think of our initial value $(R_0,G_0)$ as the genealogical distances and the types of a sequence of individuals that are drawn independently at random from an infinite population at time $0$.  Specifically, a basic assumption made throughout the paper will be that $(R_0,G_0)$ is distributed according to the marked distance matrix distribution of some $\mathbb I$-marked ultrametric measure space (as defined in Sec.~\ref{SecGen}).

Obviously this assumption implies that the pair $(R_0,G_0)$ is {\em exchangeable} in the sense that for all $n \in \mathbb N$ and all permutations $\pi$ of $[n] = \{1,\ldots,n\}$ one has
\begin{equation}\label{exch}
\big((R_0(i,j))_{1\le  i,j\le n}\, , (G_0(i))_{1\le  i\le n} \big) \stackrel d= \big((R_0(\pi(i),\pi(j))_{1\le i,j\le n}\, , (G_0(\pi(i))_{1\le i\le n}\big).
\end{equation}

Conversely, a version of the Gromov-Vershik representation theorem  (\cite[Corollary 3.12]{Gu2})
ensures that each $(R_0,G_0)$ obeying \eqref{exch} can be realized as the second step in a two-stage experiment, whose first step is the random choice of (an isomorphy class of) a marked ultrametric measure space (or equivalently of a marked distance matrix distribution), and whose second step is the marked distance matrix that arises by an i.i.d. drawing from that marked ultrametric measure space.

Let us remark that also the trival initial condition $R_0(i,j)=0$, $i,j \in \N$, together with an exchangeable~$G_0$, fits into this framework as a special case.

\subsubsection{Type configuration and type frequencies}\label{sec:type}
The process $X=(R,G)$ provides ``microscopic" information on the  type configuration and genealogies of the individuals in the lookdown system. The fluctuations of the population mass obtained from~\eqref{totmass} and the time change~\eqref{tch} deal with ``macroscopic'' quantities and are not seen directly in the lookdown representation. However both scales are coupled: we will see that the type frequencies arise from the microscopic (i.e.~individual-based) type configurations and appear in the coefficients of the SDE \eqref{totmass} whose solution in turn will impact  the local dynamics of the lookdown levels.\\

For a type configuration $g \in \mathbb I^\mathbb N$ we will say that {\em $g$ admits type frequencies} if the limiting measure
\begin{equation}\label{mug}
\mu^g:= \lim_{n\to \infty} \frac 1n \sum_{i=1}^n \delta_{g(i)}
\end{equation}
exists in the weak topology on $M^1(\mathbb I)$, the space of probability measures on $\mathbb I$ (which in our case with two traits simply means that $\mu^g\{A\}  :=  \lim_{n\to \infty} \tfrac{1}{n} \sum_{i=1}^n \mathbf 1_{\{g(i)=A\}}$ exists). We will then call $\mu^g$ the {\em type distribution belonging to} $g$.\\
We will construct the type process $G$ in such a way that it a.s. admits type frequencies at every time~$s$, hence allowing to read off the proportion $\mu^{G_s}\{A\}$ of type $A$ at time $s$ from the configuration $(G_s(i))_{i\in \N}$. \\

These proportions will play a role in the dynamics of genealogies and type configurations (see \eqref{fecdraw} and \eqref{eqG} below), and also in the SDE \eqref{totmass} for the total mass process, which in view of the time change~\eqref{tch}, becomes:
\begin{equation}\label{totmassLD}
d\zeta_s= \left(b \mu^{G_s}\{A\}\zeta^2_s - 2c\mu^{G_s}\{A\}\mu^{G_s}\{B\} \,   \zeta_s^3\right) ds + \zeta_s d \mathcal W_s,
\end{equation}
where $\mathcal W$ is a standard Brownian motion.
Similarly, \eqref{twotype} becomes
\begin{align}
\begin{split}
\label{twotypeLD}
d\zeta_s^A&= b \mu^{G_s}\{A\} \zeta_s^2 ds -c\mu^{G_s}\{A\} \mu^{G_s}\{B\} \zeta_s^3\, ds+ \sqrt{\zeta_s^A\zeta_s}\ d\mathcal W_s^{A} \\
d\zeta_s^B&= -c \mu^{G_s}\{B\} \mu^{G_s}\{A\} \zeta_s^3\,ds + \sqrt{\zeta_s^B  \zeta_s}\ d\mathcal W_s^{B},
\end{split}
\end{align}
where $\mathcal W^A$ and $\mathcal W^B$ are independent standard Brownian motions.
Possible explosion or extinction events are treated at the beginning of Section~\ref{sec:martingalepbzetaGLambda}.

\subsubsection{Pathwise construction of the lookdown process}\label{sec:pathwiseconstructionLD}
The construction of the process $(\zeta_s, G_s)_{s\in[0,\sigma)}$ is achieved via the so called `lookdown' graphical construction.
The ingredients are
\begin{itemize}
\item[(I1)]
a standard Brownian motion $\mathcal W=(\mathcal W_s)_{s\ge 0}$,
\item[(I2)]
a family $ (\LL_{ij})_{i,j \in \mathbb N, i< j}$ of independent rate 1 Poisson point processes on $\mathbb R_+$,
\item[(I3)]
a family $(\mathcal K_i)_{i \in \mathbb N}$ of independent Poisson point processes on $\mathbb R_+ \times \mathbb R_+ \times [0,1]\times\{\beta, \delta\}$ whose intensity measure is the product of the Lebesgue measure on $\mathbb R_+ \times \mathbb R_+ \times [0,1]$ and of the counting measure on $\{\beta,\delta\}$,
\end{itemize}
where the random elements 	in (I1), (I2) an (I3) are independent. Here and below we use the notation $\mathbb R_+ := [0,\infty)$. The first component of the state space of $\mathcal K_i$ is the time axis, its second and third component will serve as {\em activation level} and {\em sampling seed} at the corresponding events, and the symbols $\beta$ and $\delta$ will mark the {\em potential selective birth events} and the {\em potential competitive death events}, respectively; see the update rule \eqref{fecdraw} and the explanations preceding it, as well as the illustration in Figure~\ref{fig:lookdownwithselectionandcompetition}.   The family  $(\LL_{ij})$ can be combined to a Poisson point process $\mathcal L$ on $\mathbb R_+ \times \{(i,j)\in \N^2 : 1 \le i < j < \infty\}$, and the family  $(\mathcal K_i)$ can be combined to a Poisson point process $\mathcal K$ on $\mathbb R_+ \times \bigcup\limits_{i\in \mathbb N} \{i\} \times \mathbb R_+ \times [0,1]\times \{\beta, \delta\}$.
In this way $\LL_{ij}$ corresponds to the restriction of $\mathcal L$ to $\mathbb R_+ \times \{(i,j)\}$  and $(\mathcal K_i)$ to the restriction of $\mathcal K$ to  $\mathbb R_+ \times \{i\} \times \mathbb R_+ \times [0,1]\times \{\beta, \delta\}$. In summary, the Brownian motion $\mathcal{W}$ drives the  fluctuations of the population size, $\mathcal L$ encodes the neutral birth events, and $\mathcal K$ encodes the  potential selective  events affecting the levels in the graphical construction.

 To each atom of $\LL_{ij}$ (for $i<j$), say at time $s$, we associate an arrow starting from $(s,i)\in \R_+\times \N$ and directed to $(s,j)$. This arrow corresponds to a natural birth for the individual at level $i$ at time $s$, placing an offspring of the same type at level $j$. Levels that were equal to or above $j$ at time $s-$ (i.e.~levels $k\in \N$ such that $k\ge j$) are shifted up by~1. See Figure \ref{fig:neutralgenealogy} and cf.\ also \cite{PW06}. {This lookdown process can be seen as the limit of a finite particle system as described in Donnelly and Kurtz \cite{donnellykurtz_99} 
 where individuals with highest levels are removed at natural death events. That is why, heuristically, the natural death events are not seen any more on finite levels in the limit of infinitely many particles of small masses, as the highest level tends to infinity.
It is now the varying population mass $(\zeta_s)$ that tracks the changing mass due to demographic events.}

\begin{figure}[!ht]
\begin{center}
\includegraphics[width=17cm]{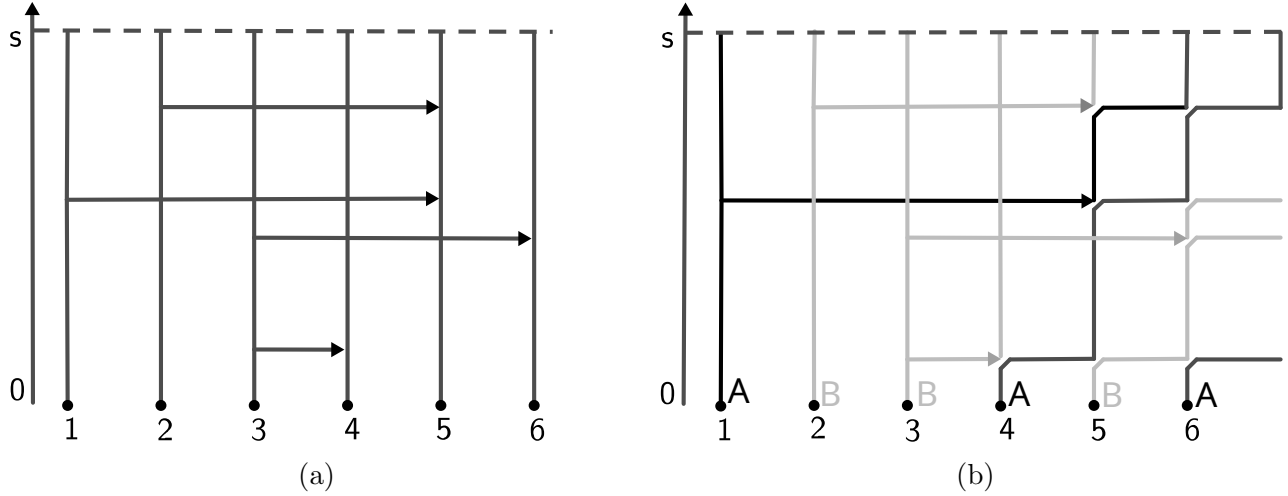}
(a) \hspace{8cm} (b)
\caption{{\small \textit{Neutral genealogy in the lookdown representation: individuals correspond to levels (abscissa), and time runs along the ordinate axis. The neutral genealogy uses only atoms of the Poisson point process $\LL_{ij}$ (for $i<j$). (a) The atoms of $\LL_{ij}$ correspond to arrows from level $i$ to level $j$. (b) Using the arrows one defines a `neutral' genealogy: an arrow from $i$ to $j$ at time $s$ corresponds to individual $i$ at time $s$ giving birth to a daughter (of the same type) at level $j$. Levels that were above $j$ at time $s-$ are shifted up by 1. Given the arrows and the type of the ancestral individuals (at the bottom of the picture), it is possible to reconstruct the types in the population at every time for the neutral genealogy by following the lineages back in time (travel along the arrows in opposite direction).}}}\label{fig:neutralgenealogy}
\end{center}
\end{figure}

To obtain the (effective) {\em selective births} and the {\em competitive deaths}, we will use a state dependent thinning of the Poisson point processes $\mathcal K_i$, to take into account the dependencies with respect to the rest of the population. As already mentioned, the marks $\beta$ and $\delta$ specify whether an atom of $\mathcal K_i$ corresponds to a potential selective birth or a potential competitive death. The variable $s$ is the time at which the atom is encountered. The probabilities with which the potential selective births and the competitive selective deaths become effective involve interactions and thus depend on the state of the process. Consequently, our pathwise construction works with an acceptance-rejection rule that uses the marks $(z,w)\in \R_+\times [0,1]$.  Before we give a definition of the rule in \eqref{fecdraw}, we explain it in words (see also Figure~\ref{fig:lookdownwithselectionandcompetition} for an illustration).
The rule is in accordance   with \eqref{twotypeLD}: 
\begin{itemize}
\item To an atom $(s,z,w)$ of $\mathcal   K_i(.  \times \{\beta\})$ with $s<\sigma$ and $z\leq b\mu^{G_{s-}}\{A\} \zeta_{s-}$ corresponds a selective birth: the individual sitting previously at level $i$ is replaced by the daughter of an individual chosen  `uniformly' from all the individuals of type $A$ that live at time $s-$.
\item
To an atom $(s,z,w)$ of $\mathcal   K_i(.  \times \{\delta\})$ with $s<\sigma$ corresponds a competitive death if $G_{s_-}(i)=B$ and $z\leq c\mu^{G_{s-}}\{A\}\zeta_{s-}^2$, or if $G_{s_-}(i)=A$ and $z\leq c\mu^{G_{s-}}\{B\}\zeta_{s-}^2$. Then the individual at level $i$ dies from the competition pressure exerted by the other type and is replaced by the daughter of an individual chosen  `uniformly' from all the individuals that live at time $s-$, irrespective of their type.
\end{itemize}
The way to sample an individual `uniformly' from all the individuals alive at a certain time
will be described in Sec.~\ref{sec:pathwiseconstructionofgenealogies}, and will be formally specified in Sec.~\ref{neut}.
Indeed we will show that it is possible to define (random) sampling measures $m^{R_{s-},G_{s-}}(d\theta,dh)$ on $\mathbb{T}\times \mathbb{I}$ where $\mathbb{T}$ will be the completion of the set of levels $\mathbb{N}$ with respect to the ultrametric $R_{s-}$ that will be described in Sec.~\ref{sec:pathwiseconstructionofgenealogies} and constructed in Sec.~\ref{sec_sel-gen}.

\begin{figure}[!ht]
\begin{center}
\includegraphics[width=10cm]{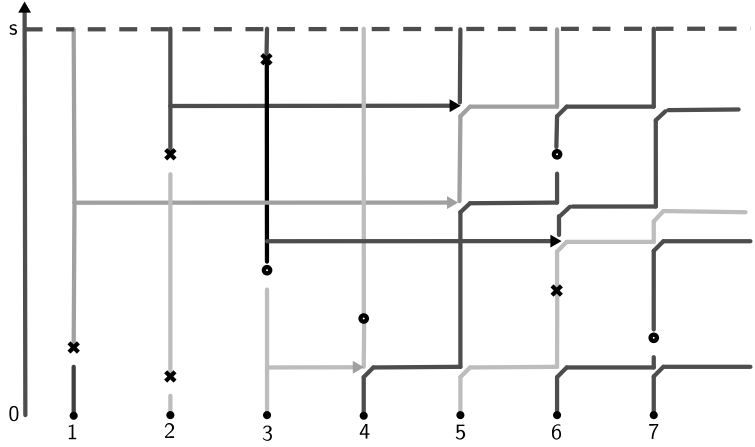}
\caption{{\small \textit{The Poisson point measures $\mathcal{K}_i$, $i\in \N$ appear on top of the neutral lookdown representation. Their atoms correspond to potential selective births (marked as $\circ$) or potential competitive deaths (marked as $\times$). Together with the individuals $(0,i)$ they form the roots of the fragments of $Z$, see Sec.\ref{RaF}. Each potential event may become active or not, according to the acceptance-rejection rule \eqref{fecdraw} explained in the text. At an active event, the individual at the existing level is replaced by the ``daughter'' of an individual  sampled at random among the $A$ individuals (selective births) or among the whole population (competitive death). 
}}}\label{fig:lookdownwithselectionandcompetition}
\end{center}
\end{figure}

To implement the acceptance-rejection rule mentioned  in the previous paragraph, we will make use of a measurable mapping $\kappa: M^1(\mathbb I) \times [0,1] \to \mathbb I$ which is such that for a  random variable $\Upsilon$ that is uniformly distributed on the interval $[0,1]$, and all $\nu \in M^1(\mathbb I)$, the random variable $\kappa (\nu, \Upsilon)$ has distribution $\nu$. This will allow us to construct  random variables with a prescribed distribution $\nu$, using  the third component of an atom of the Poisson point measures  $\mathcal{K}_i$ as input for $\kappa(\nu,.)$. Recall that this third component is uniformly distributed on $[0,1]$.\\
Our update rule for the process of type configurations $G$ works by means of a mapping
\begin{align*} q:  \mathbb I \times \mathbb I^\mathbb N \times \mathbb R_+ \times \mathbb R_+\times [0,1] \times \{\beta, \delta\} \to \mathbb I
\end{align*} which prescribes how to change the type $g(i)$ given $G_{s-}(i) =h$,  $G_{s-} = g$ and $\zeta_{s-} = v$, and given there is an atom of $\mathcal K_i$ at $(s,z, w,\beta)$ or $(s,z, w,\delta)$.
Specifically, we put for a  $g$ that admits type frequencies
\begin{align}\begin{split}\label{fecdraw} 
q(h, g,v, z, w, \beta) &:= \begin{cases}A &\qquad\mbox{ if } z \le b\, \mu^g\{A\}\,  v,  \\ h  &\qquad \mbox{ otherwise,}   \end{cases}
\\
q(h, g,v, z, w, \delta) &:= \begin{cases}\kappa(\mu^g, w) &\mbox{ if } z \le c\, \mu^g\{B\}\,  v^2 \mbox{ and } h=A,  \\ & \mbox{ or}\, \,  z \le c\, \mu^g\{A\}\,  v^2 \mbox{ and } h=B,\\
h  &\mbox{ otherwise.} \end{cases}
\end{split}
\end{align}

As in \cite[Section 4]{MR1728556} we identify the type space $\mathbb I=\{A,B\}$ with the additive group $\{0,1\}$. This corresponds to considering $\ind_{\{G_s(i)=B\}}$ instead of the type $G_s(i)$ itself and allows to formulate more easily our SDE for the process of type configurations $(G_s)$. With an initial condition $G_0$ admitting type frequencies according to \eqref{mug}, this is
\begin{align}\label{eqG}
\begin{split}
G_s(j) = G_0(j)&+ \sum_{i=1}^{j-1}\int_{[0,s]}(G_{u-}(i)-G_{u-}(j))d\LL_{ij}(u)\\ &+  \sum_{1\le i<k<j}\int_{[0,s]}(G_{u-}(j-1)-G_{u-}(j))d\LL_{ik}(u)
\\ &+\int_{[0,s]\times \mathbb R_+\times [0,1]\times\{\beta,\delta\}} (q(G_{u-}(j), G_{u-},     \zeta_{u-}, z,w, \omega) - G_{u-}(j)) d\mathcal K_j(u,z, w,\omega),
\end{split}
\end{align}
for $j \in \mathbb N$.

The following theorem characterizes the mass and type configuration process $(\zeta, G)$ with the triple $(\mathcal W, \mathcal L, \mathcal K)$ as the source of randomness, and also asserts the fact that a.s. $G$ admits type frequencies at any time.

\begin{theorem}\label{Gprop}
Let $G_0$ be exchangeable. Then the system  \eqref{totmassLD}, \eqref{eqG} for the total mass process~$(\zeta_s)$ and the type configurations $(G_s)$ has a unique strong solution, up to the possibly infinite time $\sigma$ defined in~\eqref{e:def-sigma} at which $\zeta$ goes to extinction or explodes. For this unique solution, a.\,s.\ the type frequencies $\mu^{G_s}$ and $\mu^{G_{s-}}$ (as defined in \eqref{mug}) exist for all $s\in[0,\sigma)$, i.\,e.\
\begin{equation}\label{eq:sampling}
	\mu^{G_s}\{A\}=\lim_{n\rightarrow\infty}\frac{1}{n}\sum_{i=1}^n \ind_{\{G_s(i)=A\}},\quad
	\mu^{G_{s-}}\{A\}=\lim_{n\rightarrow\infty}\frac{1}{n}\sum_{i=1}^n \ind_{\{G_{s-}(i)=A\}}.
\end{equation}
Moreover, $s\mapsto \mu^{G_s}\{A\}$ and $s\mapsto \zeta_s$ are a.\,s.\ continuous.
\end{theorem}
Here we set $0-:=0$. The proof of Theorem~\ref{Gprop} will be given in Section~\ref{sec_itscheme}, based on the preparations in Sections~\ref{neutsamp} and~\ref{RaF}. In Section~\ref{LDgen}, we will build the genealogy with selection and competition on top of the neutral genealogy. The next two subsections will explain the main ideas and tools of this construction.

\subsubsection{Filling in the ancestry: from the neutral to the  selective lookdown space}\label{sec:pathwiseconstructionofgenealogies}

With the total mass process $\zeta$ and the type configuration process $G$ being provided by Theorem~\ref{Gprop}, we can construct the ancestral lineages and ``fill in'' the process $R$ in a pathwise manner. In this subsection we will explain the graphical construction  of the process $R$ on top of  $(\zeta, G, \mathcal L, \mathcal K)$.

A crucial role will be played by a family of {\em sampling measures}. These arise as follows.  As illustrated by Figure \ref{fig:neutralgenealogy} (see Section \ref{neut} for a formal definition), the Poisson point process $\mathcal L$ together with the initial random semi-ultrametric $R_0$ give rise to a $(R_0,\mathcal L)$-measurable random semi-metric $\rho^{(0)}$ on $\mathbb R_+\times \mathbb N$, where
$\rho^{(0)}((s,i),(s',i'))$ is the genealogical distance of $(s,i)$ and $(s',i')$ in the neutral case (i.e.~without considering the atoms of the Poisson point measures $\mathcal{K}$ associated with selective births and competitive deaths). The completion of $(\mathbb R_+\times \mathbb N, \rho^{(0)})$ is denoted by $(Z,\rho^{(0)})$, and called the {\em neutral lookdown space}. The completion is done realization-wise; in this sense one should think of  $(Z,\rho^{(0)})$ as a random metric space. In slight abuse of notation, we refer by $(s, i)  \in \mathbb R_+ \times \mathbb N$ also
to the element of the metric space after the identification of elements
with distance zero and the completion, that is we also assume
$\mathbb R_+ \times \mathbb N \subset Z$ in this sense. The space $(Z,\rho^{(0)})$ describes the continuum of all individuals ever alive, together with their  distances in the neutral genealogy.\\

It is known (\cite{Gu1}, Thm 3.1) that there exists, on an event of probability 1 that does not depend on $s$, a family $(\mathfrak m_s)_{s>0}$ of probability measures on $Z$ such that
\begin{equation} \label{neutsampmeas}
\mathfrak m_s = \text{w-}\lim_{n\to \infty} \frac 1n \sum_{i=1}^n \delta_{(s,i)},
\qquad\mathfrak m_{s-} = \text{w-}\lim_{n\to \infty} \frac 1n \sum_{i=1}^n \delta_{(s-,i)},
\end{equation}
where $(s-,i):=\lim_{s'\uparrow s}(s',i)$ in $(Z,\rho^{(0)})$.
Here and below, $\text{w-}\lim$ denotes the weak limit of probability measures. In \eqref{neutsampmeas} the underlying topology 
on  $Z$ is provided by the metric $\rho^{(0)}$.
From~\cite[Thm 3.1]{Gu1}, it also follows that $s\mapsto \mathfrak m_s$ is a.\,s.\ continuous with respect to the weak topology on $(Z,\rho^{(0)})$.
The measures $\mathfrak m_s$ allow to ``sample uniformly'' from the population at time $s$, and will be called the (family of) {\em neutral sampling measures}.\\

\begin{definition} \label{defproper}a)
For a semi-ultrametric $r\in \R^{\N^2}$, we define $\mathbb T^r$ as the completion of (the set of levels) $\mathbb N$ with respect to~$r$. For $r\in\R^{\N^2}$ that is not a semi-ultrametric, we define $\mathbb T^r$ in an arbitrary way, for definiteness as $\mathbb T^r =\{1\}$.

b) Given a marked distance matrix $(r,g) \in \R^{\N^2}\times\mathbb I^\mathbb N$, we say that $(r,g)$ is {\em proper} if $r$ is a semi-ultrametric on $\mathbb N$ and
\begin{equation}\label{defmrg}
{\mathrm m}^{r,g} := \text{w-}\lim_{n\to \infty}\frac 1n \sum_{i=1}^n \delta_{(i,g(i))}
\end{equation}exists  on $\mathbb T^r\times \mathbb I$, endowed with the product topology. We will then refer to ${\mathrm m}^{r,g}$ as the {\em (marked) sampling measure} obtained from $(r,g)$.
If $(r,g)$ is not proper, we define ${\mathrm m}^{r,g}$ in an arbitrary manner, for definiteness as ${\mathrm m}^{r,g} := \delta_{(1,A)}$.

c) Let ${\mathrm m}^r =  \text{w-}\lim_{n\to \infty}\frac 1n \sum_{i=1}^n \delta_i$ denote the projection of ${\mathrm m}^{r,g}$ to $\mathbb{T}^r$. If, for a proper marked distance matrix $(r,g)$, there exists a measurable function $\bar g: \mathbb T^r\to \mathbb I$ such that
\begin{equation}\label{markf}
{\mathrm m}^{r,g}(d(\theta,h)) = {\mathrm m}^r(d\theta)\delta_{\bar g(\theta)}(dh),\qquad \theta \in \mathbb T^r, h \in \mathbb I,
\end{equation}
we will speak of $\bar g (\theta)$ as {\em the type carried by the individual $\theta$}. Note also that in this case, the type frequencies $\mu^g\{A\},\ \mu^g\{B\}$ correspond to the projection of ${\mathrm m}^{r,g}$ on the type component $h\in \mathbb{I}$.
\end{definition}

In Section~\ref{genealogy}, we will extend the concept of the neutral lookdown space  to our present setting  by constructing a {\em selective lookdown space}.  Here is a short preview. On each level $i$, selective births and competitive deaths can occur  only at the discrete time points given by $\mathcal K_i$.
This discreteness allows to dissect the neutral lookdown space into countably many {\em fragments}, {\em rooted} in those points  $(s,j)$ that carry atoms of $\mathcal K$ or belong to $\{0\} \times \mathbb N$. Each fragment consists of the completion of all the lineages descending from the ancestor $(s,j)$ until a selective birth or a competitive death affects them. Hence, each fragment is monotypic, inheriting the type of its root, whose type, in turn, is determined by $G$ from Theorem \ref{Gprop} (see Figure~\ref{fig:lookdownwithselectionandcompetition} for an illustration).

To describe all individuals ever alive by a {\em connected} metric space, we continue the ancestral lineages backwards in time until they hit a root, that is an atom of $\mathcal{K}$), say at time $s$. We say that a root is \textit{active} if the activation condition (i.e. the corresponding first condition of \eqref{fecdraw}) is fulfilled (otherwise, the selective fecundity event or the competition event proposed by $\mathcal{K}$ do not happen). If a competition event occurs at time $s$, then choose the parent individual, i.e. the individual  that continues the ancestral lineage backwards in time, according to  $\mathfrak m_s$.  If a fecundity event occurs at time $s$, the individual which reproduces is drawn according to $\mathfrak m_s$ conditioned on the fragments being of type~$A$. This yields, up to the time change given by the total mass process $\zeta$ as indicated in the next paragraph, a random metric space $(\hat Z,\rho)$, which is our selective lookdown space, see  Definition~\ref{fromrhotoR}.

This definition will also specify the genealogical distances $R_s$ as a function of the metric $\rho$ and the initial distances $R_0$. Here is an informal description: In the definition of $R_s(i_1,i_2)$,  $i_1,i_2\in \mathbb N$, we distinguish two cases. In case the
 ancestral lineages of $(s, i_1)$ and  $(s, i_2)$ meet between times $0$ and $s$, we transverse the geodesic from $(s,i)$ to $(s,j)$ in $(\hat Z,\rho)$ with speed $1/\zeta_{s-u}$ when passing through an ancestor that is time $u$ back from $s$. The duration one needs to pass through this geodesic is then equal to $R_s(i,j)$.  In case the
 ancestral lineages of $(s, i_1)$ and  $(s, i_2)$ do not meet between times $0$ and~$s$, then their levels at time~$0$ are different. Denoting these levels by $a_1$ and $a_2$, we then obtain $R_s(i_1, i_2)$, by adding the distance $R_0(a_1, a_2)$ to the sum of the durations to reach $(0,a_1)$ from $(s,i_1)$ and $(0,a_2)$ from~$(s,i_2)$.

\subsubsection{Updating the distance matrix at neutral and selective reproduction events}\label{updateR}

In this subsection we describe the updating rule of $(R,G)$ at the time  of a reproductive event, i.e. at a time $s$ at which the process $\mathcal L$ or the process $\mathcal K$ has an atom.
First we consider the neutral events.
When an atom of $\LL_{ij}$ ($i<j$) is encountered at time $s$, the individual at level $i$ puts a daughter at level $j$, pushing the levels previously above and including $j$ up by 1.
We define the corresponding update $\vartheta_{i,j}(r,g)=(\vartheta_{i,j}(r),\vartheta_{i,j}(g))$ of a marked distance matrix $(r,g)\in\R^{\N^2}\times\I^\N$  by putting for $i<j$ and $\ell \in \N$
\begin{align}\label{updateg}
\begin{split}
 (\vartheta_{i,j}(g))(\ell)=
\begin{cases}
g(\ell), &\ell<j,\\
g(i), &\ell=j,\\
g(\ell-1), &\ell>j,
\end{cases}
\end{split}
\end{align}
and by defining $(\vartheta_{i,j}(r))(\ell,m)$ as the function symmetric in $\ell, m \in \N$ such that $(\vartheta_{i,j}(r))(\ell,\ell)=0$ and for $\ell < m$
\begin{align}\label{updater}
\begin{split}
(\vartheta_{i,j}(r))(\ell,m)=
\begin{cases}
r(\ell,m),		&1\leq \ell<m<j,\\
r(\ell,i),		&1\leq \ell<j=m,\\
r(\ell,m-1),	&1\leq \ell<j<m,\\
r(i,m-1),	&\ell=j<m,\\
r(\ell-1,m-1),	&j<\ell<m.
\end{cases}
\end{split}
\end{align}
In a selective birth or a competitive death, the individual at some level $j \in \mathbb N$ is replaced by another individual  from the closure of the present population. Specifically, for a marked distance matrix $(r,g)\in\R^{\N^2}\times\I^\N$, let $\mathbb{T}$ be the completion  of $(\N,r)$ with respect to $r$, and let $\theta \in \mathbb{T}$, $h'\in \mathbb{I}$. Then the corresponding update ${\tilde\vartheta}_{j,\theta,h'}(r,g)=(\tilde\vartheta_{j,\theta}(r), \tilde\vartheta_{j,h'}(g))$ is done by putting
\begin{align}\label{selupdate}
(\tilde\vartheta_{j,h'}(g))(\ell)=
\begin{cases}
g(\ell) &\mbox{if }\ell \neq j,\\
h' &\mbox{if }\ell=j,
\end{cases}
\\\label{selupdate2}
(\tilde\vartheta_{j,\theta}(r))(\ell,m)=
\begin{cases}
r(\ell,m) 	&\mbox{if }\ell,m \neq j,\\
r(\ell,\theta) 		&\mbox{if }\ell\neq j, m=j,
\end{cases}
\end{align}
and $\tilde\vartheta_{j,\theta}(\ell,m)$ as symmetric in $\ell,m$, and $\tilde\vartheta_{j,\theta}(\ell,\ell)=0$.

With the rule just described, we can read off the jumps of the marked distance matrix process $(R,G)$ at those times $s$ which are charged by the Poisson point process $\mathcal K$. Let us explain how the changes are parameterized by the variables attached to the atom of $\mathcal{K}$ at time $s$.\\

If $\mathcal K_j$ has an atom in $(s,z,w,\delta)$, then we need to pick an individual from  ${\mathrm m}^{R_{s-}, G_{s-}}$ (where we remark that $(R_{s-},G_{s-})$ is a.\,s.\ proper, see Corollary~\ref{RGproper}).\\
Given $(r,g)\in\mathbb R^{\mathbb N^2}\times \mathbb I^{\mathbb N}$, the pick of an individual from $\mathrm  m^{r,g}$ can be obtained from a measurable mapping $w \mapsto (\theta^{r,g}(w), h^{r,g}(w))$ from $[0,1]$ to $\mathbb T^r\times \mathbb I$ (see Def. \ref{defproper}) which transports the uniform distribution on $[0,1]$ into $\mathrm  m^{r,g}$. We can now specify the $\kappa(\mu^g,w)$ appearing in \eqref{fecdraw} as
\begin{equation}\label{compcons}
\kappa(\mu^g,w)=h^{r,g}(w).
\end{equation}
Then:
\begin{align} \label{updatedelta}
(R_s,G_s) = \begin{cases} \tilde \vartheta_{j, \theta^{R_{s-},G_{s-}}(w), h^{R_{s-},G_{s-}}(w)}(R_{s-},G_{s-})   &\mbox{ if } z\leq c \, \zeta_{s-}^2\, \mu^{G_{s-}}\{A\} \mbox{ and } G_{s-}(j) = B\\
&\mbox{ or } z\leq c \, \zeta_{s-}^2\, \mu^{G_{s-}}\{B\} \mbox{ and } G_{s-}(j) = A,\\
(R_{s-},G_{s-}) & \mbox{otherwise.}
\end{cases}
\end{align}

If $\mathcal K_j$ has an atom in $(s,z,w,\beta)$, we need to sample an individual from $\mathrm  m^{r,g}(d\theta,dh\mid h= A)$. Let $w \mapsto (\tilde\theta^{r,g}(w), \tilde h^{r,g}(w))$ be a measurable mapping from $[0,1]$ to $\mathbb T^r\times \mathbb I$ which transports the uniform distribution on $[0,1]$ into $\mathrm  m^{r,g}(d\theta,dh\mid h= A)$. Notice that here, we necessarily have $\tilde h^{r,g}(w)=A$. Then,
\begin{align} \label{updatebeta}
(R_s,G_s) = \begin{cases} \tilde \vartheta_{j, \tilde\theta^{R_{s-},G_{s-}}(w), A}(R_{s-},G_{s-})   &\mbox{ if } z\leq b \, \zeta_{s-}\mu^{G_{s-}}\{A\}\\
(R_{s-},G_{s-}) & \mbox{otherwise}.
\end{cases}
\end{align}




In a nutshell, we can embed all the preceding updating rules into a single SDE:
\begin{multline}\label{eqRG}
R_s(i,j)= R_0(i,j)+
2\int_{[0,s]}\zeta_u dR_u(i,j)
+\sum_{1\leq k<\ell\leq j}
\int_{[0,s]} \big( \vartheta_{k,\ell}(R_{u-})(i,j)-R_{u-}(i,j)\big) d\mathcal L_{k\ell}(u)\\
\begin{aligned}+\sum_{k\in \{i,j\}} \int_{[0,s]} &\mathbf 1_{\{G_{u-}(k)=B, z\leq c\zeta_{u-}^2\mu^{G_{u-}}\{A\},\omega=\delta  \} }  \,\\
&\times\big( \tilde\vartheta_{k,\theta^{R_{u-},G_{u-}}(w),h^{R_{u-},G_{u-}}(w)}(R_{u-},G_{u-})(i,j) - R_{u-}(i,j)\big)
d\mathcal K_k(u,z,w,\,\omega)\\
+\sum_{k\in \{i,j\}} \int_{[0,s]}& \mathbf 1_{\{G_{u-}(k)=A, z\leq c\zeta_{u-}^2\mu^{G_{u-}}\{B\} ,\, \omega=\delta \} }  \,\\
&\times\big( \tilde\vartheta_{k,\theta^{R_{u-},G_{u-}}(w),h^{R_{u-},G_{u-}}(w)}(R_{u-},G_{u-})(i,j) - R_{u-}(i,j)\big)
d\mathcal K_k(u,z,w,\omega)\\
+\sum_{k\in \{i,j\}} \int_{[0,s]} &\mathbf 1_{\{z\leq b\zeta_{u-}\mu^{G_{u-}}\{A\},\,  \omega=\beta\} }  \,\\
&\times\big( \tilde\vartheta_{k,\tilde\theta^{R_{u-},G_{u-}}(w),\tilde h^{R_{u-},G_{u-}}(w)}(R_{u-},G_{u-})(i,j) - R_{u-}(i,j)\big)
d\mathcal K_k(u,z,w,\omega),
\end{aligned}
\end{multline}
where we write $(r',g')(i,j):=r'(i,j)$. In the light of the above constructions, the following result is now an immediate consequence of Theorem \ref{Gprop}.
\begin{theorem}\label{zetaRG}
Let $(R_0,G_0)$ be distributed according to the marked distance matrix distribution of some $\I$-marked metric measure space. Then the system  \eqref{totmassLD}, \eqref{eqG}, \eqref{eqRG} of SDE's  has a unique strong solution  $(\zeta_s, R_s, G_s)_{s\in[0,\sigma)}$ up to the stopping time $\sigma$ defined in~\eqref{e:def-sigma}, and this process $(\zeta, R, G)$ is Markovian.
\end{theorem}

\subsubsection{A well-posed martingale problem for the evolving lookdown genealogy}
\label{sec:martingalepbzetaGLambda}
Let $(\mathcal W, \mathcal L, \mathcal K)$ be as in Sec. \ref{sec:pathwiseconstructionLD}, and let $(\zeta,R,G)$ be the process provided by Theorem \eqref{zetaRG}. The process $\zeta$ can touch zero or explode in finite time; this happens on the event $\sigma < \infty$ with $\sigma$ being defined in \eqref{e:def-sigma}. The time $\sigma$ is announced by the following sequence of stopping times $\sigma_M$, $M\in \mathbb N$:
 \begin{equation}\label{sigmaM}
\sigma_M:=\inf\big\{s\geq 0,\ \zeta_s\notin (1/M,M)\big\}.
\end{equation}
With $b$ and $c$ being the parameters that appear in \eqref{twotype}, we set
\begin{align}\label{CM1}
C_M:= (b\vee c)M^2.
\end{align}
 Let us now introduce the state space for the process $(\zeta,R,G)$ stopped at $\sigma_M$. We define:
 \begin{align}\label{EMlambda1}
E_M := \left( \left(\tfrac 1M, M\right)\times \mathbb R^{\mathbb N^2}\times  \mathbb I^\mathbb N\right) \cup \{\Delta_M\}
\end{align}
where  $\left(\tfrac 1M, M\right)\times  \mathbb R^{\mathbb N^2} \times \mathbb I^\mathbb N$ is equipped with the product topology and where $\Delta_M$ is a cemetery point such that a sequence $(v_n,r_n,g_n)$ of $E_M$ is said to converge to $\Delta_M$ if either $v_n \to \frac 1M$ or $v_n \to M$ as $n \to \infty$.\\

Next we display the generator of $(\zeta, R, G)$ restricted to appropriate test functions $F$.
For $n\in\N$, let $\rho_n:\R_+\times\R^{\N^2}\to\R_+\times\R^{n^2}$, $(v,r)\mapsto (v,(r(i,j))_{1\leq i,j\leq n})$, be the restriction map.
We define $D_{1,M}$ as the set of those functions $f:(\frac 1M,M)\times\mathbb R^{\N^2} \to \R$ for which there exists an $n\in\N$,  a compact set $C \subset (\frac 1M,M)$, and a bounded infinitely differentiable function
$\phi :(\frac 1M,M)\times \R^{\N^2} \to \R$, all whose derivatives are bounded, 
such that   $\phi(v, r) = 0$ unless $v\in C$, and $f=\phi\circ\rho_n$.

Let $D_{2}$ be the set of those functions $\gamma: \mathbb I^{\mathbb N}\to \mathbb R$ for which there exists an $n \in \mathbb N$ such that $\gamma(g)$ depends only on the first $n$ coordinates of $g$.

We now consider functions $F:E_M\to\R$ of the form
\begin{align} \label{firstF1}
	F(v,r,g) =f(v,r) \gamma(g)
\end{align}
for $(v,r,g)\in E_M$ with $v\in(1/M,M)$,
where $f \in D_{1,M}$, $\gamma \in D_2$. Here we also assume that $F$ is continuous in $\Delta_M$. The smallest possible $n \in \mathbb N$ which fits to the required representations of $f$ and $\gamma$ will be called the {\em degree} of $F$. We write $F_{r(i,j)}$ for the partial derivative of $F$ with respect to the variable $r(i,j)$, and $F_{v}$ for partial derivatives of $F$ with respect to $v$.

Let $\vartheta_{i,j}$ and $\tilde \vartheta_{j,\theta, h'}$ be as in \eqref{updateg}, \eqref{updater},  \eqref{selupdate} and  \eqref{selupdate2}.
For a pair $(r,g)$, let ${\rm m}^{r,g}$ be the marked sampling measure  as in \eqref{defmrg}, and $\mu^g$ be the second marginal of ${\rm m}^{r,g}$, which is equal to the type distribution belonging to $g$.
Let $F$ be as in  \eqref{firstF1} with degree $n$ and $(v,r,g)\in E_M$ with $v\in (1/M,M)$. Then we define $\mathbf AF$  as follows:
     \begin{align}
     \begin{split}
\label{genX}
\mathbf AF(v,r,g)&=\frac{v^2}{2}F_{vv}(v,r,g)
+\big(bv^2\mu^g\{A\}-2cv^3\mu^g\{A\}\mu^g\{B\}\big)F_v(v,r,g)\\
&+2v\sum_{1\leq i\neq j\leq n}F_{r(i,j)}(v,r,g)\\
&+\sum_{1\leq i<j\leq n}(F(v,\vartheta_{i,j}(r,g))-F(v,r,g))\\
&+cv^2\mu^g\{A\}\sum_{j=1}^n\int_{\mathbb{T}^r\times \mathbb I} {\mathrm m}^{r,g}(d\theta,dh')
\mathbf 1_{\{g(j)=B \} }
(F(v,{\tilde\vartheta}_{j,\theta,h'}(r,g))-F(v,r,g))\\
&+cv^2\mu^g\{B\}\sum_{j=1}^n\int_{\mathbb{T}^r\times \mathbb I} {\mathrm m}^{r,g}(d\theta,dh')
\mathbf 1_{\{g(j)=A \} }
(F(v,{\tilde\vartheta}_{j,\theta,h'}(r,g))-F(v,r,g))\\
&+bv\sum_{j=1}^n\int_{\mathbb{T}^r\times \mathbb I} {\mathrm m}^{r,g}(d\theta,dh') \mathbf 1_{\{h'=A\}} (F(v,{\tilde\vartheta}_{j, \theta,h'}(r,g))-F(v,r,g)),
\end{split}
\end{align}
and we set $\mathbf A F(\Delta_M)=0$.
Let $D_M$ be the linear span of the constant real-valued functions on $E_M$ (defined in \eqref{EMlambda1}) and all functions of the form  \eqref{firstF1}. The linear extension of \eqref{genX}  to $D_M$ will again be denoted by~$\mathbf A$.
\begin{proposition}\label{MPA} For all $M > 0$, the
 process $(\zeta_{s\wedge \sigma_M}, R_{s\wedge \sigma_M},G_{s\wedge \sigma_M})_{s \ge 0}$  solves the martingale problem $(\mathbf A, D_M)$, and this martingale problem is well-posed.
\end{proposition}
The proof of Proposition \ref{MPA}  will be given in Section \ref{twomp}. This proof will heavily rely on Theorem~\ref{zetaRG} but the uniqueness part will need additional arguments.
The first one of these (Proposition \ref{MPAlambda}) will be to establish a well-posed martingale problem for a refinement $(\zeta, R, G, \Lambda)$, where $\Lambda$ counts the points of $\mathcal L$ and $\mathcal K$, and thus keeps track of all the essential graphical ingredients that are needed to specify the jump distribution of $R$ at these points. The second step will complete the proof of Proposition \ref{MPA} by applying Kurtz' Markov Mapping Theorem, thus projecting to a well-posed martingale problem for the first three components  $(\zeta, R, G)$. Let us also mention that a similar strategy has been applied in Lemma~4.2 in \cite{MR1728556} in a situation without the components $\zeta$ and $R$, i.e. for a dynamics with constant population size and without consideration of the genealogies.

\subsection{A well-posed martingale problem for the evolving  symmetrized genealogy}\label{mppopsizeandgenealogy}
Let $(\zeta_s, R_s, G_s)$, $s<\sigma$, be the process provided by Theorem \ref{zetaRG}. In Corollary \ref{RGproper} we will prove that a.s.
$(R_s, G_s)$ is proper in the sense of Definition \ref{defproper}.   From this, we define a process of marked genealogies (see Sec. \ref{SecGen}) whose state at each time $s$ is the isomorphy class of the marked ultrametric  measure space $(\mathbb{T}^{R_s},R_s, \mathrm{m}^{R_s,G_s})$ (again see  Definition \ref{defproper}). Recalling the notation $X_s =  (R_s,G_s)$, we denote this isomorphy class by~$\psi(X_s)$.\\
For $\chi \in \mathbb M$, the space of marked genealogies (see Sec. \ref{SecGen}), we will write $\nu^\chi$ for the {\em marked distance matrix distribution} obtained from $\chi$, i.e. the distribution of $(R^\chi,G^\chi)$ where $R^\chi$ is the distance matrix and $G^\chi$ is the type configuration of a sequence drawn i.i.d. from the sampling measure belonging to (an arbitrary representative of)~$\chi$.
For a prescribed initial condition $(v_0, \chi_0) \in (0,\infty)\times \mathbb M$, we define
$X_0:=(R^{\chi_0},G^{\chi_0})$ and take $(v_0, X_0)$  as  initial condition for the process $(\zeta, X)$. Let the time change $t(s)= \int_0^s \zeta_u\, du$, $s< \sigma$, be as in \eqref{inversetch}, and let $t \mapsto s(t)$ be its inverse. For $t$ such that $s(t)<\sigma$ we define
\begin{align}\label{def:psi}
(\xi_t, Y_t):=  (\zeta_{s(t)},\psi (X_{s(t)})).
\end{align}
We now set out to describe the process $(\xi,Y)$ by a stopped martingale problem. That $\xi$ can reach zero or converge to infinity may be problematic for the change of time \eqref{tch}. That is why it is natural to introduce, for a fixed positive integer $M>0$, the stopping time
\begin{equation}
\label{eq:tauM}
\tau_M=\inf\{t\in\R_+, \xi_t\notin(1/M,M)\}
\end{equation}
and the stopped processes $\xi^{\tau_M}=\xi_{\cdot\wedge\tau_M}$ and $Y^{\tau_M}=Y_{\cdot\wedge\tau_M}$.
Let us also define $\tau_0 := \lim_{M\to\infty}\tau_M$, and
choose $M$ depending on the initial condition $v_0$ of the mass process so large that $v_0\in \left(\frac 1M, M\right)$.

 For $F\in D_M$ (defined just after \eqref{genX}), $v > 0$ and  $\chi \in \mathbb M$ we put
 \begin{equation}
 \label{eq:PhiF_APhiF}
 \Phi_F(v,\chi):= \int F(v,r,g) \nu^\chi(dr,dg), \qquad \mathbb A\Phi_F(v,\chi) :=  \int \frac 1v \mathbf A F(v,r,g) \nu^\chi(dr,dg).
 \end{equation}

In analogy to \eqref{EMlambda1} we now consider the state space
\begin{equation}\label{DefSM}
S_M := \left(\left(\tfrac 1M,M\right) \times \mathbb M\right) \cup \left\{\Delta_M\right\},
\end{equation}
where $\left(\tfrac 1M,M\right) \times \mathbb M$ is equipped with the product topology and a sequence $(v_n, \chi_n)$ is said to converge to $\Delta_M$ if either $v_n \to \frac 1M$ or $v_n\to M$ as $n\to \infty$. In other words, this corresponds to a ``lumping'' of all states $(\frac 1M,\chi)$ and $(M,\chi)$ with $\chi\in\M$ into one state $\Delta_M$.

\begin{theorem}\label{martprobth}
For $(v_0, \chi_0)\in (0,\infty)\times \mathbb M$ and $M > 1/v_0$,
the process $(\xi^{\tau_M},Y^{\tau_M})$ is Markovian and gives the unique solution of the martingale problem
\begin{align}\label{martprobY}
(\xi_0,Y_0)=(v_0, \chi_0), \qquad
\Phi_F(\xi^{\tau_M}_t,Y^{\tau_M}_t) -\int_0^{t\wedge{\tau_M}}  \mathbb A\Phi_F(\xi_u, Y_u)\, du = \text{ martingale},  \qquad    F \in D_M.
 \end{align}
 \end{theorem}
Theorem \ref{martprobth} will be proven in Section \ref{symmgen}.
\begin{remark} Since $\tau_M \uparrow \tau_0$ a.s. as $M\to \infty$, the process $(\xi_t, Y_t)_{t<\tau_0}$ is characterized in distribution by the requirement that, when stopped at $\tau_M$, it solves the martingale problem \eqref{martprobY} for all $M\in \mathbb N$.
\end{remark}

In contrast to $(X_s)_{s<\sigma}$, which has jumps, the process $(Y_t)_{t<\tau_0}$ is continuous. This is contained in the next result, which will be proved in Section~\ref{symmgen}.  
 \begin{proposition}
 \label{propcont}
 For $(v_0,\chi_0)\in(0,\infty)\times\mathbb M$ and $M\in\mathbb N$, the process $(\xi_{t\wedge \tau_M},Y_{t\wedge \tau_M})_{t\geq 0}$ has a.s. continuous paths in $\mathbb R_+\times\mathbb M$.
 \end{proposition}
%
If $(\xi, Y)$ is the solution of \eqref{martprobY}, then $\xi$ is a weak solution of \eqref{totmass}, as can be seen immediately by projecting \eqref{martprobY} to its first  component.  We can also recover the equations for $\xi^A$ and $\xi^B$ given in \eqref{twotype}. For this, we first recall that the type frequencies $\mu_t\{A\}$ and $\mu_t\{B\}$ can be recovered from the projection of ${\mathrm m}^{Y_t}(d(\theta,h))$ on its second component. Consequently, $(\xi^A_t,  \xi^B_t)$ is defined in terms of $(\xi, Y)$ as $(\xi_t \mu_t\{A\}, \xi_t \mu_t\{B\})$. 

 \begin{proposition}\label{prop:projection}
Let $(\xi, Y)$ be the solution of \eqref{martprobY}. Then  $(\xi^A_t, \xi^B_t)_{t\geq 0}$ is a weak solution of the SDE~\eqref{twotype}.
 \end{proposition}
 The proof will be given in Section~\ref{proof_of_two_props}. Because of Proposition \ref{strongsol_type_process}, $\xi$ is non-explosive, hence we infer that $\tau_0=\inf\{t\geq 0,\ \xi_t=0\}$ (with the usual convention that $\inf \emptyset = \infty$).

\section{Building blocks from neutrality}\label{neut}
\subsection{Neutral lookdown space and marked sampling measures}\label{neutsamp}
The \textit{neutral setting} will provide the building blocks for the analysis of the genealogy also in the presence of selection and competition, and we study it specifically in this section. Its only ingredients are the initial condition $(R_0,G_0)$ (being distributed according to the marked distance matrix distribution of a marked ultrametric measure space), and the neutral birth events given by the Poisson point  measures $\{\LL_{i,j}, 1\leq i<j\}$ (see (I2) in Section~\ref{secmain}). As illustrated by Figure \ref{fig:neutralgenealogy}, each of the points of $\LL_{i,j}$ can be seen as a {\em merger} of two ancestral lineages: if  $\LL_{i,j}$ has an atom at time $s$, then the ancestral lineage of $(s,j)$ starts, back into the past, from $(s-,i)$, from there on being identical with the ancestral lineage of $(s,i)$. In this case, the (neutral) genealogical distance of $(s,i)$ and $(s,j)$ equals zero; more generally, the neutral genealogical distance of $(s_1,i)$ and $(s_2,j)$ is determined as follows: trace the neutral ancestral lineages back from $(s_1,i)$  and $(s_2,j)$. If they merge at time $u \ge 0$, then the distance is $(s_1-u)+(s_2-u)$. Otherwise, if $a_1$ and $a_2$ are the labels of the two neutral ancestors at time  $0$, the distance is $s_1+s_2+R_0(a_1,a_2)$. For given $R_0$, this gives rise to an $\mathcal L$-measurable random semi-metric $\rho^{(0)}$ on $\mathbb R_+ \times \mathbb N$. The {\em neutral lookdown space} is the metric completion of $(\mathbb R_+ \times \mathbb N, \rho^{(0)})$,  denoted by $(Z, \rho^{(0)})$. It carries the family of {\em sampling measures}  $\m_s$, $s>0$, defined by \eqref{neutsampmeas}.

By the Glivenko-Cantelli lemma, the assumption that $(R_0,G_0)$ has the marked distance matrix distribution of a marked ultrametric measure space ensures that a.s.
\begin{equation}
\label{eq:ass-m0}
\text{w-}\lim_{n\to\infty}\frac{1}{n}\sum_{i=1}^n\delta_{(i,G_0(i))} \quad \text{exists on } \mathbb T_0\times \mathbb I,
\end{equation}
where $\mathbb T_0$ is  the metric completion of $(\N,R_0)$. This clearly implies that
\begin{equation}
\label{eq:ass-m00}
\m_0 := \text{w-}\lim_{n\to\infty}\frac{1}{n}\sum_{i=1}^n\delta_{(0,i)}
\end{equation}
exists on $Z$, including the case $s=0$ into \eqref{neutsampmeas}. Likewise, \eqref {eq:ass-m0} implies that
%
\begin{equation}\label{m0}
\mathrm{m}_0:=\text{w-}\lim_{n\to\infty}\frac{1}{n}\sum_{i=1}^{n}\delta_{((0,i),G_0(i))}
\end{equation}
exists on $Z\times\I$ a.\,s.
With the notation introduced in Definition \ref{defproper}, and since $R_0$ is a random semi-ultrametric, \eqref{m0} says that $(R_0,G_0)$ is a.s.\ proper. \\

In the next lemma we show,  based on the existence of the neutral sampling measure \eqref{neutsampmeas},  that  the corresponding statement also holds true for $s>0$.
For the neutral genealogy, and for a time $s\geq 0$, denote by $G_s^{(0)}(i)$ the type of $(s,i)$ given by
\begin{equation}\label{neuttrans}
G_s^{(0)}(i):= G_0(a),
\end{equation}
where $a$ is the level of the neutral ancestor of $(s,i)$ at time $0$.
\begin{lemma}
\label{lem:m0s}
The weak limits
\begin{equation} \label{propertypes}
\mathrm {m}_s^{(0)}:=\text{w-}\lim_{n\to\infty}\frac{1}{n}\sum_{i=1}^{n}\delta_{((s,i),G^{(0)}_s(i))},\quad
\mathrm {m}_{s-}^{(0)}:=\text{w-}\lim_{n\to\infty}\frac{1}{n}\sum_{i=1}^{n}\delta_{((s,i),G^{(0)}_{s-}(i))}
\end{equation}
exist on $Z\times\I$ (endowed with the product topology) on an event of probability 1 that does not depend on $s$. Moreover,
$s\mapsto \mathrm {m}_s^{(0)}$ is a.\,s.\ continuous with respect to the weak topology on $Z\times\I$.
\end{lemma}
\begin{proof}
Fix $s>0$.
The map $\hat{f}: \{s\}\times\N\to\I$, $(s,i)\mapsto G^{(0)}_s(i)$, is uniformly continuous with respect to~$\rho^{(0)}$. To see this, let $\delta<2s$ arbitrary and suppose $\rho^{(0)}((s,i_1),(s,i_2)) < \delta$. Then $(s, i_1)$ and $(s,i_2)$ have a common ancestor at time $s-\delta/2$; consequently their types coincide.

Thus the map $\hat{f}$ can be extended to a (uniformly) continuous function $\hat{f}:\overline{\{s\} \times \N}\to\I$, where $\overline{\{s\} \times \N}$ denotes the closure of $\{s\} \times \N$ with respect to $\rho^{(0)}$. Then the map $f:\overline{\{s\}\times\N}\to \overline{\{s\}\times\N}\times\I$, $\theta\mapsto(\theta,\hat f(\theta))$ is also continuous. It satisfies
\[\frac{1}{n}\sum_{i=1}^n\delta_{((s,i),G^{(0)}_s(i))}
=\left(\frac{1}{n}\sum_{i=1}^n\delta_{(s,i)}\right) \circ f^{-1},\]
where the right hand side denotes the image measure under $f$. With $\mathfrak m_s$ being the neutral sampling measure defined in \eqref{neutsampmeas}, the  continuous mapping theorem implies that $\frac{1}{n}\sum_{i=1}^n\delta_{((s,i),G^{(0)}_s(i))}$ converges weakly to $\mathrm{m}_s^{(0)}=\m_s \circ f^{-1}$ on an event of probability $1$ that does not depend on $s$. The assertions on the left limits and on continuity follow similarly using the left limit in~\eqref{neutsampmeas} and continuity of $(\mathfrak m_s)$.
\end{proof}

\begin{remark}
\label{rmk:type-distr-neutral}
The function $\hat f$ is called mark function in the sense of e.g.\ \cite{KliemLoehr2015}. The $Z$-component  of $m_s^{(0)}$ is $\m_s$ given by  \eqref{neutsampmeas}, and the  $\I$-component of $m_s^{(0)}$ is the type distribution $\mu^{(0)}_s$ under the neutral transport, which in view of \eqref{propertypes} and \eqref{mug}  obeys a.s.
\begin{equation}\label{muG}
\mu^{(0)}_s={\rm m}_s^{(0)}(Z\times\cdot)=  \text{w-}\lim_{n\to\infty}\frac{1}{n}\sum_{i=1}^{n}\delta_{G^{(0)}_s(i)}  .
\end{equation}
\end{remark}

\subsection{Partitioning the lookdown space: roots and fragments}\label{RaF}


Throughout the article, we assume  that our initial state $(R_0,G_0)$ has the marked distance matrix distribution of a marked ultrametric measure space. This implies in particular that $G_0$ admits type frequencies a.s. (see eq. \eqref{m0}). Moreover, we then have the marked neutral sampling measures \eqref{propertypes} at hand for all times $s$.  Recall the families of independent Poisson point measures $\{\mathcal K_i: i \in \N\}$ from ingredient (I3). We now partition (up to a set that is not charged by any of the sampling measures $\mathfrak m_s$) the entire space $Z$ into (what we call) \textit{fragments} $\Gamma_\gamma \subset Z$ with \textit{roots} $\gamma \in \R_+ \times \N$ as follows. On top of the neutral lookdown construction, we think of a competition respectively a fecundity ``cross'' added at $(s,i)$ when  $\mathcal K_i$ places an atom at $(s,i)$. To be more precise, fix $C<\infty$ and put
\begin{align}\label{FV}
  \widetilde{\mathcal R} &:= \bigcup_{i\in \mathbb N}  \Big(\big\{ s \in \R_+: \mathcal K_i(\{s\} \times [0,C] \times [0,1]\times\{\beta, \delta\}) \geq 1 \big\} \times \{i\} \Big)\subset \R_+ \times \N.
\end{align}
Note that the restriction to $[0,C]$ guarantees that the overall rate of potential events on a fixed level is bounded on any finite time-interval. As a consequence the points in $\widetilde{\mathcal R} \cap (\R_+ \times \{i\})$ do not accumulate for fixed $i\in \N$ almost surely.

Now let
\begin{equation}\label{defRoots}
 \mathcal R := \widetilde{\mathcal R} \cup (\{0\} \times \mathbb N) \subset \R_+ \times \N 
\end{equation}
be the set of \textit{roots}. The types and lineages in the subtree above a root evolve according to the dynamics of the neutral model until they hit another root.

\begin{remark}\label{neutallin}
Note that the (neutral) ancestral lineage of each element $\theta \in Z \backslash (\R_+ \times \N)$ is well-defined. Indeed, take a sequence $(s_n, \ell_n) \in \R_+ \times \N,\, n \in \N$, such that $(s_n,\ell_n) \rightarrow \theta$ for $n \rightarrow \infty$. Then we have in particular $s_n \rightarrow s \in \R_+$. Without loss of generality, assume $(s_n)_{n \in \N}$ is monotonically increasing. How to determine the ancestor of $\theta$ at time $s-\epsilon$ for $0 < \epsilon < s$ arbitrary? There exists $n_0 \in \N$ such that $\rho^{(0)}((s_n,\ell_n),\theta) < \epsilon$ for all $n \geq n_0$, that is, $(s_n,\ell_n)$ and $\theta$ have a common ancestor at time $s-\epsilon$ for all $n \geq n_0$. Thus, take the ancestor of $\theta$ at time $s-\epsilon$ to be the one of $(s_{n_0},\ell_{n_0})$.
\end{remark}

For each $\gamma \in \mathcal R$ let
\begin{align*}
  \Gamma_\gamma :=
  \{ \theta \in Z : & \ \theta \mbox{ descends from } \gamma \mbox{ and there are no points in } \mathcal R \mbox{ on the } \\
  & \mbox{ lineage connecting } \theta \mbox{ with } \gamma \} \subset Z.
\end{align*}

\begin{remark}
We make the following observations.
\begin{enumerate}
\item[1)]
Interpret $\Gamma_\gamma$ as descendants of $\gamma$ in a \textit{neutral} infinite alleles model with mutation. Here, the frequencies exist and $Z$ can be broken into a countable number of fragments, rooted in $\mathcal{R}$. This construction yields $\m_s(\Gamma_\gamma)$ for \textit{all} times $s \geq 0$. It can be shown that a.\,s.,
$$ \m_s\big( \bigcup_{\gamma \in \mathcal{R}} \Gamma_\gamma \big) = 1 $$
for all $s>0$. Further details are given in the proof of Lemma \ref{lem:mass} below.
\item[2)]
By restricting to $\{s\}  \times \N $, a partition of $\N$ is inherited. This
partition  depends on $s$, where $i \sim j$ if $(s,i)$ and $(s,j)$ have a common ancestor living between times $0$ and $s$ and there is no root on their geodesics.
\item[3)]
In contrast to a tree-valued process whose states describe genealogical
trees at fixed times, the lookdown space describes all individuals which
live at any time. From this object, we can read off the state of the
tree-valued processes at time $s$ using a restriction of the lookdown
space. The lookdown space itself however is universal for all $s$.
\end{enumerate}
\end{remark}

The sets $\Gamma_\gamma$, $\gamma \in \mathcal R$, form a  partition of the set that is obtained from $Z$ by removing the accumulation points of $\mathcal R$ in $Z$. Almost surely, the set of these accumulation points has zero mass under all $\m_s$, $s>0$. This is the content of the following lemma.
\begin{lemma}
\label{lem:mass}
Almost surely, $\m_s(Z\setminus\bigcup_{\gamma\in\mathcal R} \Gamma_\gamma)=0$ and  $\m_{s-}(Z\setminus\bigcup_{\gamma\in\mathcal R} \Gamma_\gamma)=0$ for all $s>0$.
\end{lemma}
\begin{proof}
The points of $\mathcal R$ can be thought of as mutation events in an infinite alleles model that come with rate $C$ along the lineages  in a lookdown model, say with type space $[0,1]$ and parent independent mutation where the type in each mutation event is drawn uniformly and independently. To obtain a contradiction, assume that there exists an $s$ such that the set of accumulation points of $\mathcal R$ has nonzero mass under $\m_s$. As all such accumulation points have different types in the infinite alleles model, this results in a type distribution of mass smaller than $1$. The lookdown construction for the infinite alleles model \cite[Theorem 3.2]{MR1681126} shows, however, that there are a.\,s.\ no exceptional time points with defective type distribution.
\end{proof}

\begin{corollary}\label{corol:lemma1.8}
Almost surely,
\[\m_s(\Gamma_\gamma)= \lim_{n\to\infty}\frac{1}{n}\sum_{i=1}^n \ind_{\{(s,i)\in\Gamma_\gamma\}}=
\m_{s-}(\Gamma_\gamma)= \lim_{n\to\infty}\frac{1}{n}\sum_{i=1}^n \ind_{\{(s-,i)\in\Gamma_\gamma\}}\]
for all $\gamma\in\mathcal R$  and $s > 0$.
\end{corollary}

\begin{proof}
Almost surely, for each $s>0$ and $\gamma\in\mathcal R$, the boundary of $\Gamma_\gamma$ as a subset of $\overline{\{s\}\times\N}\subset Z$ is not charged by $\m_s$. Hence, the Portmanteau theorem,~\eqref{neutsampmeas}, and the continuity of $(\m_s)$ yield the result.
\end{proof}

\begin{lemma}\label{lem37}
\label{lem:cont-frag-masses}
Almost surely,
\begin{itemize}
\item[(i)]
the fragment masses $\m_s(\Gamma_\gamma)$ are continuous in $s$ for each $\gamma$,
\item[(ii)]
for each fragment $\Gamma_\gamma$, the restriction $\m_s(\cdot \cap \Gamma_\gamma)$ is continuous in $s$ with respect to the weak topology on $(\Gamma_\gamma,\rho^{(0)})$.
\end{itemize}
\end{lemma}

\begin{proof}
The statement (i) follows by relating the assertion to an infinite alleles model,  similar as in the proof of Lemma~\ref{lem:mass}.

To prove (ii), note that a discontinuity at a time $s$ implies the existence of a closed subset $A$ of $(\Gamma_\gamma,\rho^{(0)})$ with $\limsup_{s' \rightarrow s} \m_{s'}(A) > \m_s(A)$.
Since $\m_s(\overline{\Gamma_\gamma}\setminus \Gamma_\gamma) = 0$ by Lemma~\ref{lem:mass}, the inequality holds also for the closure of $A$ in $(Z,\rho^{(0)})$.
Thus, a discontinuity at time $s$ results in a discontinuity of $\m_s$ on the neutral lookdown space in contradiction to \cite[Theorem~3.1]{Gu1}.
\end{proof}


\begin{lemma}\label{colouringlem}
For all $\epsilon, T>0$ there exists almost surely a random $\ell \in \mathbb N$ such that
\[ \sum_{\gamma=(u,i) \in \mathcal{R},\ \mbox{ \footnotesize{with} } u \in \mathbb R_+,\, i \leq \ell} \m_s(\Gamma_\gamma) \geq 1-\epsilon \mbox{ for all } s \in [0,T]. \]
\end{lemma}

\begin{proof}
For fixed $s \in [0,T]$, this follows from Lemma~\ref{lem:mass}. Let
\[ \vartheta_k = \inf\big\{s \in [0,2T]: \sum_{\gamma=(u,i) \mbox{ \footnotesize{with} } u \in \mathbb R_+,\, i \leq k} \m_s(\Gamma_\gamma) < 1-\epsilon \big\}, \]
where we set $\inf \emptyset = 2T$. Then $\vartheta_k$ is monotonically increasing in $k$. Set
\[ \vartheta = \sup_{k \in \mathbb N} \vartheta_k. \]
It suffices to show that $\vartheta=2T$ almost surely.

On the event that $\vartheta<2T$, there exists by Lemma~\ref{lem:mass} a.\,s.\ $k\in\N$ with
\[\sum_{\gamma=(u,i) \mbox{ \footnotesize{with} } u \in \mathbb R_+,\, i \leq k} \m_\vartheta(\Gamma_\gamma) \geq  1-\epsilon/2.\]
By Lemma~\ref{lem:cont-frag-masses} (i), almost surely, the fragment masses $\m_s(\Gamma_\gamma)$ are continuous in $s$ for each $\gamma$. Hence, there exists $\delta>0$ with
\[ \sum_{\gamma=(u,i) \mbox{ \footnotesize{with} } u \in \mathbb R_+,\, i \leq k} \m_s(\Gamma_\gamma) \geq 1-\epsilon \]
for all $s\in(\vartheta-\delta,\vartheta+\delta)$. This implies that for all $\ell\geq k$ with $\vartheta_\ell>\vartheta-\delta$, we have $\vartheta_\ell\geq\vartheta+\delta$, in contradiction to the definition of $\vartheta$.
Thus $\{\vartheta < 2T\}$ must be a null event and the claim follows.
\end{proof}

\section{An SDE for type configuration and population size: proof of Theorem \ref{Gprop}}\label{sec_itscheme}
In this section we provide an iteration scheme which leads to the proof of Theorem \ref{Gprop}. We will be guided by the proof of Theorem 4.1 in \cite{MR1728556}. The additional (and substantial) challenge that is overcome in our proof is that the total mass, which in \cite{MR1728556} was assumed constant, now is a stochastic process which depends on the type configurations.

Recalling the ingredients from Section \ref{sec:pathwiseconstructionLD}, we will work with the filtration $\mathscr F = (\mathscr F_s)$, where $\mathscr F_s$ is generated by $\mathcal W_u$, $u\le s$ and those points in $\mathcal L$ and $\mathcal K$ whose time component is at most $s$.
Following the steps described in Section \ref{section:LDrepresentation}, we will prove the existence and uniqueness of the type process $G$ in \eqref{eqG}. A substantial difficulty is that the SDEs for $G$ depend on the mass process $\zeta$ (in the lookdown time-scale) that itself depends on the process $\mu^{G_.}\{A\}$ of proportions of type $A$. Also, Theorem \ref{Gprop} asserts that $G_s$  admits type frequencies for all times $s>0$, so that $\mu^{G_s}\{A\}$ is well-defined.

Let us introduce the following function, describing the drift of the process $\xi$:
\begin{equation}
\label{fabbr}
f(v, p) = bpv-2cp(1-p)v^2, \qquad v \ge 0,\,  p\in [0,1].
\end{equation}
Fix a  constant $C\in (0,\infty)$ in \eqref{FV} that bounds the rate at which selective and competitive events occur.
For a modification of the system of SDEs \eqref{eqG} and \eqref{totmassLD}, where we use $C$ and $M$ to control the dynamics (see third term in the r.h.s. of \eqref{eqGCM} below), we prove existence and strong uniqueness by a Picard iteration-like argument. For this we put
\begin{equation}
\label{frest}
f_M(v, p) = f((v\vee\tfrac 1M)\wedge M,p).
\end{equation}

The following key proposition treats SDEs similar to the ones in Theorem \ref{Gprop}, but with $f_M$ instead of $f$, which simplifies the problem of controlling the population size. Replacement of $f_M$ by $f$ will be treated at the end of the section, in the {\em completion of the proof of Theorem \ref{Gprop}}.

\begin{proposition}\label{propCM}
Let the  $\mathbb I^\mathbb N$-valued random variable $G_0$ be exchangeable. For given $M > 1$ let \mbox{$v_0 \in (\frac 1M, M)$.}   The following system \eqref{eqGCM}, \eqref{eqzeta}  of SDEs has a unique strong solution:
\begin{align}
\notag G_s(j) &= G_0(j)+ \sum_{i=1}^{j-1}\int_{[0,s]}(G_{u-}(i)-G_{u-}(j))d\LL_{ij}(u) \\
&+  \sum_{1\le i<k<j}\int_{[0,s]}(G_{u-}(j-1)-G_{u-}(j))d\LL_{ik}(u) \label{eqGCM} \\
\notag &+\int_{[0,s]\times [0,C] \times [0,1]\times\{\beta,\delta\}} (q(G_{u-}(j), G_{u-},     (\zeta_{u-}\vee \frac{1}{M}) \wedge M, z,w, \omega) - G_{u-}(j)) d\mathcal K_j(u, z,w,\omega), \\
 \zeta_0&=v_0, \quad \quad
d\zeta_s = \zeta_sf_M(\zeta_s,  \mu^{G_s}\{A\}) ds  +  \zeta_s \, d\mathcal W_s, \label{eqzeta}
 \qquad j\in\N,\quad s\ge 0.
\end{align}
For this unique solution, a.\,s.\ the type frequencies $\mu^{G_s}$ and $\mu^{G_{s-}}$ (as defined in \eqref{mug})
exist for all $s\geq 0$. Moreover, $s\mapsto\mu^{G_s}\{A\}$ and $s\mapsto \zeta_s$ are a.\,s.\ continuous.
\end{proposition}

In order to prepare the proof of this proposition, we first show a statement on the continuous dependence of \eqref{eqzeta} on its input $\mu^{G_s}\{A\}$.
 \begin{lemma}\label{Goetz} Let $\mathcal W$ be an $(\mathcal F_s)$-adapted Brownian motion and let $\varphi, \tilde{\varphi}$ be $(\mathcal F_s)$-adapted, $[0,1]$-valued and continuous.  Let the $(0,\infty)$-valued processes $\alpha, \tilde{\alpha}$ obey
\begin{align*}
d\alpha_s &= \alpha_sf_M(\alpha_s,\varphi_s)ds  + \alpha_s\, d\mathcal W_s, \\
d\tilde{\alpha}_s &= \tilde{\alpha}_sf_M(\tilde{\alpha}_s,\tilde{\varphi}_s)ds  + \tilde{\alpha}_s\, d\mathcal W_s,
\end{align*}
with the same initial condition $v_0$ in $(\tfrac 1M, M)$ at time $0$. Then there exists a constant $\tilde C$ (depending on $M$ but not depending on $s$) such that for all $s \ge 0$ we have
\begin{equation}\label{alpha}
|(\alpha_s\vee \tfrac 1M)\wedge M-(\tilde{\alpha}_s\vee \tfrac 1M)\wedge M|  \le \tilde Cse^{\tilde C s} \int_0^s |\varphi_u-\tilde{\varphi}_u| du \quad \mbox{a.s.}
\end{equation}
\end{lemma}
\begin{proof}[Proof of Lemma \ref{Goetz}] By It\^o's formula,
\[d \ln \alpha = \frac 1\alpha d\alpha -\frac1{2\alpha^2}d[\alpha] = f_M(\alpha,\varphi)ds+d\mathcal W-\frac 12 ds,\]
\[d \ln \tilde{\alpha} = \frac 1{\tilde{\alpha}}  d\tilde{\alpha} -\frac1{2\tilde{\alpha} ^2}d[\tilde{\alpha} ] = f_M(\tilde{\alpha},\tilde{\varphi})ds+d\mathcal W-\frac 12 ds.\]
Subtracting one equation from the other and using the triangle inequality we get
\begin{equation}\label{est}
|\ln \alpha_s- \ln \tilde{\alpha}_s|  \le \int_0^s |f_M(\alpha_u, \varphi_u)-f_M(\tilde{\alpha}_u, \tilde{\varphi}_u) |  du.
\end{equation}
There exists a constant $c_1$ (depending on $M$ but not depending on $s$) such that for all $s\ge 0$
\begin{equation}\label{lnX}
|(\alpha_s\vee \tfrac 1M)\wedge M -(\tilde{\alpha}_s\vee \tfrac 1M)\wedge M| \le c_1|\ln ((\alpha_s \vee \tfrac{1}{M}) \wedge M) - \ln ((\tilde{\alpha}_s \vee \tfrac{1}{M}) \wedge M)|.
\end{equation}
From  \eqref{est}, \eqref{lnX} and the Lipschitz property of $f_M$ we obtain
\[|(\alpha_s\vee \tfrac 1M)\wedge M -(\tilde{\alpha}_s\vee \tfrac 1M)\wedge M|\le c_2\int_0^s(|(\alpha_u\vee \tfrac 1M)\wedge M-(\tilde{\alpha}_u\vee \tfrac 1M)\wedge M|+|\varphi_u- \tilde{\varphi}_u|)du.\]
Using Gronwall's inequality we arrive at
\begin{equation*}
|(\alpha_s\vee \tfrac 1M)\wedge M-(\tilde{\alpha}_s\vee \tfrac 1M)\wedge M|
\le \tilde Cse^{\tilde C s} \int_0^s|\varphi_u- \tilde{\varphi}_u|du.
\end{equation*}
\end{proof}

Most of the remainder of this section is devoted to the proof of Proposition \ref{propCM} which uses an iteration scheme. To get this scheme started, we take $G_s^{(0)}$ as the neutral type transport defined by \eqref{neuttrans} and $\mu_s^{(0)}$ as the neutral type distributions given by \eqref{muG}. Let us emphasize that by de Finetti's theorem the assumption of exchangeability of $G_0$ implies that $G_0$ a.s. admits type frequencies. \\

\noindent \textbf{Step 1, Recursion hypothesis: } For $k\geq 1$, assume  that for $\ell = 0,\ldots, k-1$  we have defined $\mathscr F$-adapted $\mathbb I^\mathbb N$-valued processes $G^{(\ell)}$ and  continuous $\R_+$-valued processes $\zeta^{(\ell)}$ such that:
\begin{itemize}
\item $G^{(\ell)}_0=G_0$, and almost surely, $G^{(\ell)}_s$ and $G^{(\ell)}_{s-}$ admit type frequencies $\mu_s^{(\ell)}$ for all $s>0$, i.e. the probability measures on $\mathbb I$
\[\mu_s^{(\ell)}= \text{w-}\lim_{n\to \infty} \frac 1n \sum_{i=1}^n \delta_{G^{(\ell)}_s(i)},\quad
\mu_{s-}^{(\ell)}= \text{w-}\lim_{n\to \infty} \frac 1n \sum_{i=1}^n \delta_{G^{(\ell)}_{s-}(i)}\]
exist and $s\mapsto \mu_s^{(\ell)}$ is continuous,
\item for given $\mu^{(\ell)}$, the process $\zeta^{(\ell)}$ is the unique strong solution of the SDE
\begin{equation}\label{SDEk} \zeta_0^{(\ell)}=v_0, \quad \quad
d\zeta_s^{(\ell)} = \zeta_s^{(\ell)}f_M(\zeta_s^{(\ell)},  \mu_s^{(\ell)}\{A\}) ds  +  \zeta_s^{(\ell)}  d\mathcal W_s,
\end{equation}
and a.\,s. continuous.
\end{itemize}
The fact that for given $\mu^{(\ell)}$ the SDE \eqref{SDEk} indeed has a unique strong solution with continuous paths follows e.g. from \cite[Theorem 5.3]{protter1977}.

\noindent \textbf{Step 2, Setting up the iteration step:} In order to define $G^{(k)}$ in terms of $G^{(k-1)}$, $\zeta^{(k-1)}$, $\mathcal L$ and $\mathcal K$, we consider the following system of SDE's where the function $q$ (cf. \eqref{fecdraw}) uses the type frequencies $\mu^{(k-1)}$ which are well-defined by our recursion hypothesis.
\begin{align}\label{eqGk}
\notag
G_{s}^{(k)}(j) &= G_0(j)+ \sum_{i=1}^{j-1}\int_{[0,s]}(G_{u-}^{(k)}(i)-G_{u-}^{(k)}(j))d\LL_{ij}(u)\\ &+  \sum_{1\le i<l<j}\int_{[0,s]}(G_{u-}^{(k)}(j-1)-G_{u-}^{(k)}(j))d\LL_{il}(u)
\\ \notag &+\int_{[0,s]\times [0,C] \times [0,1]\times\{\beta,\delta\}} (q(G_{u-}^{(k)}(j), G_{u-}^{(k-1)},     (\zeta_{u-}^{(k-1)}\vee \frac{1}{M}) \wedge M, z,w, \omega) - G_{u-}^{(k)}(j)) d\mathcal K_j(u, z,w,\omega).
\end{align}
This has the following interpretation. While the type transport through the neutral lookdown events (given by the points of $\mathcal L$) happens as usual, the activation levels (appearing in the update rule \eqref{fecdraw}) for the potential selective events (given by those points $(u,z,w,\omega)$ of $\mathcal K$ with $z\le C$) are controlled by the mass process and the type frequencies {\em from the previous iteration}. Notice that $G_{u-}^{(k)}(j)$, that is the type  {\em in the current iteration}, enters as the first argument in the update rule $q$, which is relevant at a competitive death event. This amounts to having a frozen environment for the competition. Also, notice that we use the mass process $(\zeta_{u-}^{(k-1)}\vee \frac{1}{M}) \wedge M$ truncated at $M$ and $1/M$. For the sequel, let us define the first time at which the truncation is effective: $\sigma^{(\ell)}_M=\inf\Big\{s\geq 0,\ \zeta_s^{(\ell)}\in \big\{\frac{1}{M},M\big\}\Big\}$.

We now use \eqref{eqGk} to successively update the types of $\gamma \in
\mathcal R_{>0} := \{\gamma =(s,i) \in \mathcal R \mid s >0\}$ in the $k$-th iteration, where $\mathcal R$ is defined in \eqref{defRoots}. This we do by first recording all the
roots $(s_0,i_0)$, \ldots, $(s_n,i_n)\in \mathcal R$  that lie on the neutral ancestral lineage of $\gamma$, with $0 = s_0 < \cdots < s_n=s$. The type of $(s_0,i_0)$ remains to be $G_0(i_0)$; the new type $G_{s_1}^{(k)}(i_1)$ of $(s_1,i_1)$ is determined by taking $G_0(i_0)$ as the first argument in the update rule $q$,  the new type $G_{s_2}^{(k)}(i_2)$ of $(s_2,i_2)$ is  determined by taking $G_{s_1}^{(k)}(i_1)$ as the first argument of the update rule $q$, etc.

Having thus re-colored all $\gamma \in \mathcal R_{>0}$ in the $k$-th iteration, we complete the recoloring by letting $\Gamma_\gamma$ inherit the type of its root, i.e. by setting, for each $(s,j)\in \mathbb R_+\times \mathbb N$,  its type $G_{s}^{(k)}(j)$ equal to the type of that $\gamma$ for which $(s,j) \in \Gamma_\gamma$.

For $s > 0$ and  $h\in \mathbb I$ we now put
\begin{equation}\label{defmk}
\mu_s^{(k)}\{h\} := \sum_{\gamma =(u,i) \in \mathcal R: G_u^{(k)}(i) = h} \m_s(\Gamma_\gamma)
\end{equation}
and note that $s\mapsto \mu_s^{(k)}\{h\}$ is a.\,s.\ continuous as a consequence of Lemmas~\ref{lem37} and~\ref{colouringlem}.
Recall \eqref{neutsampmeas}: $\m_s(\Gamma_\gamma)$ is the weight which the neutral sampling measure at time $s$ assigns to that part of the neutral offspring of $\gamma$ whose ancestral lineages  are not separated from $\gamma$ by some other root.

The next assertion, which will also be used in the uniqueness part (Step 4 of the proof of Proposition \ref{propCM}), and will therefore be singled out as a lemma, shows that $G^{(k)}$ a.s. admits type frequencies at all times.

\begin{lemma}
\label{lem:desc-mu}
In each iteration step $k=1,2,\ldots$ we have a.s.%
\begin{equation}
\label{equ:mu-equ-lim}
  \mu_s^{(k)} = \text{w-}\lim_{n \rightarrow \infty} \frac{1}{n} \sum_{i=1}^n \delta_{G_s^{(k)}(i)}  = \text{w-}\lim_{n \rightarrow \infty} \frac{1}{n} \sum_{i=1}^n \delta_{G_{s-}^{(k)}(i)}       \qquad \mbox{ for all } s>0.
\end{equation}
\end{lemma}
\begin{proof}
Let $s,\epsilon>0$ be arbitrarily fixed, and take $h \in \mathbb I$.
On an a.s. event that does not depend on $s$, there exists by Lemma~\ref{lem:mass} a finite set $\{\gamma_1,\ldots,\gamma_\ell\}\subset\mathcal R$ of roots such that
\[\sum_{j=1}^\ell\m_s(\Gamma_{\gamma_j})>1-\epsilon.\]
By \eqref{neutsampmeas} and Corollary \ref{corol:lemma1.8}, it follows that on an a.s. event that does not depend on $s$,
\[\lim_{n\to\infty}\frac{1}{n}\sum_{i=1}^n\ind_{\left\{(s,i)\in\bigcup_{j=1}^\ell\Gamma_{\gamma_j}\right\}}>1-\epsilon,\quad
\lim_{n\to\infty}\frac{1}{n}\sum_{i=1}^n\ind_{\left\{(s-,i)\in\bigcup_{j=1}^\ell\Gamma_{\gamma_j}\right\}}>1-\epsilon.\]
For $j=1,\ldots,\ell$, we also write the roots more explicitly as $\gamma_j=:(u_j,r_j)$. Then,
for all iterations $k$, we have
\begin{align*}
\limsup_{n\to\infty}\frac{1}{n}\sum_{i=1}^n\ind_{\left\{G_s^{(k)}(i)=h\right\}}
&\leq\limsup_{n\to\infty}\frac{1}{n}\sum_{i=1}^n \left( \sum_{j=1}^\ell\ind_{\left\{(s,i)\in\Gamma_{\gamma_j},G_{u_j}^{(k)}(r_j)=h\right\}}+\ind_{\left\{(s,i)\notin\bigcup_{j=1}^\ell\Gamma_{\gamma_j}\right\}} \right)\\
&\leq\sum_{j=1}^\ell\ind_{\left\{G_{u_j}^{(k)}(r_j)=h\right\}}\lim_{n\to\infty}\frac{1}{n}\sum_{i=1}^n\ind_{\left\{(s,i)\in\Gamma_{\gamma_j}\right\}}+\epsilon\\
&\leq\mu_s^{(k)}\{h\}+\epsilon,
\end{align*}
where we used the definition \eqref{defmk} of $\mu_s^{(k)}$.
Similarly,
\begin{align*}
\liminf_{n\to\infty}\frac{1}{n}\sum_{i=1}^n\ind_{\left\{G_s^{(k)}(i)=h\right\}}
&\geq\sum_{j=1}^\ell\ind_{\left\{G_{u_j}^{(k)}(r_j)=h\right\}}\lim_{n\to\infty}\frac{1}{n}\sum_{i=1}^n\ind_{\left\{(s,i)\in\Gamma_{\gamma_j}\right\}}\\
&\geq\sum_{\gamma=(u,i)\in\mathcal R}\ind_{\left\{G_{u}^{(k)}(i)=h\right\}}\lim_{n\to\infty}\frac{1}{n}\sum_{i=1}^n\ind_{\left\{(s,i)\in\Gamma_{\gamma}\right\}}-\lim_{n\to\infty}\frac{1}{n}\sum_{i=1}^n\ind_{\left\{(s,i)\notin\bigcup_{j=1}^\ell\Gamma_{\gamma_j}\right\}}\\
&\geq\mu_s^{(k)}\{h\}-\epsilon.
\end{align*}
The assertion on the left limits follows analogously using the continuity of $\mu^{(k)}$ as defined in~\eqref{defmk}.
\end{proof}
Summarizing the results so far, we are able to define the values $G^{(k)}_s(j)$ and $\mu^{(k)}_s$ for $j\in \N$ and $s>0$.
Using Equation~\eqref{SDEk} we can define $\zeta_s^{(k)}$, $s>0$.\\

\noindent \textbf{Step 3, Convergence of the iteration scheme: }
For two type vectors $g$ and $g'$ admitting type frequencies we have (with $|| \mu^{g} - \mu^{g'}||$ denoting the total variation distance of $\mu^g$ and $\mu^{g'}$)
\begin{equation}
\label{gmestimate}
|| \mu^{g} - \mu^{g'}|| \le \limsup_{n\to \infty} \frac 1n
\sum_{i=1}^n \mathbf 1_{\{g(i) \neq g'(i)\}}.
\end{equation}

Denoting by $\lambda$  the ``Lebesgue times uniform'' measure on $B_C:= [0,C]\times [0,1] \times \{ \beta, \delta\}$, we see from the definition of $q$ in \eqref{fecdraw} that for all $g, g'$ admitting type frequencies, all $v, v'\ge 0$  and all $h\in \mathbb I$  we have
\begin{align}\label{Lip1}
\int_{B_C} &\mathbf 1_{\{q(h, g,v,a)\neq q(h, g',v',a)\}}\lambda(da) \notag\\
&\le C(b|\mu^g\{A\}v-\mu^{g'}\{A\}v'| +c|\mu^g\{A\}v^2-\mu^{g'}\{A\}(v')^2|)+c|\mu^g\{B\}v^2-\mu^{g'}\{B\}(v')^2|).
\end{align}
Using \eqref{gmestimate}, we infer that for all $g, g'$ admitting type frequencies, all $v, v' \ge 0$ and all $h, h'\in \mathbb I$,
\begin{align}\label{Lip}
\int_{B_C} &\mathbf 1_{\left\{q(h, g,(v \vee \tfrac 1M)\wedge M,a)\neq q(h', g',(v' \vee \tfrac 1M)\wedge M,a)\right\}}\lambda(da) \notag\\
&\le D\left(\mathbf 1_{\{h\neq h'\}} + \limsup_{n\to \infty} \frac 1n\sum_{i=1}^n \mathbf 1_{\{g(i) \neq g'(i) \}} + |v-v'|\right),
\end{align}
where the constant $D$ may depend on $M$ and $C$.  (Note that this is an analogue of \cite[(4.14)]{MR1728556}.)

Fix $T\in \mathbb R_+$ and $j\in \mathbb N$. For $0\le s \le T$ and $j\in \mathbb N$, let us define by $\mathcal A_s^T(j)$ the level of the ancestor of $(T,j)$ at time $s$ in the neutral genealogy. Note that $\mathcal A^T(j)$ is $\mathcal  L$-measurable, obeying the SDE
\begin{align} \label{4.4DK99b}
\mathcal A^T_s(j) & = j - \sum_{1\leq i< \ell <j} \int_{(s,T]} \ind_{  \{\mathcal A^T_u(j) >\ell \} } d\LL_{i\ell}(u) \nonumber \\
& \quad - \sum_{1\leq i< \ell \le j} \int_{(s,T]} (\ell-i)\ind_{  \{\mathcal A^T_u(k) =\ell \} } d\LL_{i\ell}(u).
\end{align}
For
$s>0$, we abbreviate $\tilde G_{ s}^{T,k}(j) := G^{(k)}_s(\mathcal A_s^{T}(j))$.

Let the point measure $\tilde {\mathcal K}_j$ be defined such that $\tilde {\mathcal K}_j$ has an atom in $(u, a)$ if and only if $ {\mathcal K}_{\mathcal A^T_u(j)}$ has an atom in $(u,a)$, $0\le u\le T$, $a \in B_C$.
For notational reasons we now consider the induction step from $k$ to $k+1$ instead of $k-1$ to $k$.
We  get as an analogue to the estimate starting at p. 1112 line~-3 in \cite{MR1728556}, that for $s\in(0,T]$:
\begin{align} 
&\ind_{ \{ \tilde G_{ s-}^{T,k+1}(j) \neq \tilde G_{ s-}^{T,k}(j) \}  }
\leq \int_{ [0,s] \times B_C} (  1 - \ind_{ \{ \tilde G_{ u-}^{T,k+1}(j)\neq \tilde G_{ u-}^{T,k}(j) \}  }) \nonumber \\ 
&\qquad \times \ind_{\{q( \tilde G_{ u-}^{T,k}(j) , G^{(k)}_{u-} , (\zeta^{ (k)}_{u-}\vee \frac{1}{M}) \wedge M , a ) \neq q( \tilde G_{ u-}^{T,k}(j) , G^{(k-1)}_{u-} , (\zeta^{ (k-1)}_{u-} \vee \frac{1}{M}) \wedge M, a ) \}} \tilde{\mathcal{K}}_j(du,da) \nonumber \\ 
& = \int_{ [0,s] \times B_C} (  1 - \ind_{ \{ \tilde G_{ u-}^{T,k+1}(j)\neq \tilde G_{u-}^{T,k}(j) \}  }) \nonumber \\ 
&  \quad\times \ind_{\{q( \tilde G^{T,k}_{u-}(j) , G^{(k)}_{u-} , (\zeta^{ (k)}_{u-}\vee \frac{1}{M}) \wedge M, a ) \neq q( \tilde G_{u-}^{T,k}(j) , G^{(k-1)}_{u-} , (\zeta^{ (k-1)}_{u-}\vee \frac{1}{M}) \wedge M, a ) \}} \nonumber
(\tilde{\mathcal{K}}_j(du,da) -  du\, \lambda(da))   \nonumber \\ 
& \qquad + \int_{ [0,s] \times B_C} (  1 - \ind_{ \{ \tilde G^{T,k+1}_{u-}(j)\neq \tilde G^{T,k}_{u-}(j) \}  }) \nonumber \\ 
&  \quad\times \ind_{\{q( \tilde G_{u-}^{T,k}(j) , G^{(k)}_{u-} , (\zeta^{ (k)}_{u-}\vee \frac{1}{M})\wedge M , a ) \neq q( \tilde G^{T,k}_{u-}(j) , G^{(k-1)}_{u-} , (\zeta^{ (k-1)}_{u-}\vee \frac{1}{M})\wedge M , a ) \}} du\, \lambda(da) \nonumber \\ 
& \leq \text{martingale} + D\int_0^s \eta^{(k)}_udu + D \int_0^s  | (\zeta^{(k)}_{u}\vee \frac{1}{M})\wedge M - (\zeta^{(k-1)}_{u} \vee \frac{1}{M})\wedge M| du, 
\label{backcomp}
\end{align}
where
 \begin{align}\label{rho}
 \eta^{(k)}_u:= \lim_{n\to \infty} \frac 1n \sum_{i=1}^n \ind_{\{G^{(k)}_{u-}(i)\neq G^{(k-1)}_{u-}(i)\}},
\end{align}
and where we used \eqref{Lip} in the last estimate. To see that the limit exists we can argue as in the proof of Lemma \ref{lem:desc-mu}. Indeed, the limit equals the sum of the masses at time $u$ of the fragments whose roots are colored differently at iterations $k$ and $k-1$.
Taking expectations of both sides in the above  estimate, putting $s:= T$   and noting that $\tilde G_{T}^{T,k}(j) = G^{(k)}_{T}(j)$,
we obtain the estimate
\begin{equation}\label{firstest}
\mathbf P(G_T^{(k+1)}(i)\neq G_T^{(k)}(i)) \le D\left(\mathbf E(\int_0^{T} \eta^{(k)}_u du) + \mathbf E(\int_0^{T}|(\zeta^{(k)}_{u} \vee \frac{1}{M})\wedge M -(\zeta^{(k-1)}_{u} \vee \frac{1}{M})\wedge M| du)\right),
\end{equation}
which in turn implies
\begin{equation}
\mathbf E( \eta^{(k+1)}_{T}) \le D
\left( \int_0^T \mathbf E(  \eta^{(k)}_{u} ) du + \int_0^T \mathbf E( |(\zeta^{(k)}_{s} \vee \frac{1}{M})\wedge M -(\zeta^{(k-1)}_{s} \vee \frac{1}{M})\wedge M|) ds\right), \label{eq:etape1}
 \end{equation}
 by dominated convergence.
Equation \eqref{eq:etape1} and Lemma \ref{Goetz} (with $\alpha=\eta^{(k+1)}$, $\tilde\alpha=\eta^{(k)}$, $\varphi=\mu^{(k)}$, $\tilde\varphi=\mu^{(k-1)}$) give that
\begin{equation*}
\mathbf E( \eta^{(k+1)}_{T}) \le D
\left( \int_0^T \mathbf E(  \eta^{(k)}_{u} ) du + \tilde CT e^{\tilde C T} \int_0^T \int_0^s \mathbf E(  |\mu^{(k)}_{u}\{A\}  - \mu^{(k-1)}_{u}\{A\}| )du\ ds\right).
 \end{equation*}
Using \eqref{gmestimate} we arrive at
\begin{equation}\label{new14Agust30}
\mathbf E( \eta^{(k+1)}_T) \le D \int_0^T \mathbf (1+\tilde CT^2e^{\tilde C T}) \mathbf E(  \eta^{(k)}_{u} ) du
\end{equation}
for all $T \geq 0$. By a direct reiteration, this gives:
\begin{align*}
\mathbf  E\big(\eta_T^{(k+1)}\big) \leq & D (1+\tilde CT^2e^T) \int_0^T \left[ D (1+\tilde Cs_1^2e^{s_1}) \int_0^{s_1} \mathbf E \big(\eta_{s_2}^{(k-1)}\big)ds_2\right] ds_1 \\
\leq & (D (1+\tilde CT^2e^{\tilde C T}))^2 \int_0^T \int_0^{s_1}\mathbf E\big(\eta_{s_2}^{(k-1)}\big)ds_2 \ ds_1\\
\leq & (D (1+\tilde CT^2e^{\tilde C T}))^k \int_0^T \int_0^{s_1}\dots \int_0^{s_{k-1}} \mathbf E\big(\eta_{s_k}^{(1)}\big)ds_k \dots  ds_1.
\end{align*}
Because $\mathbf E\big(\eta_{s}^{(1)}\big)$ is uniformly bounded by 1, we obtain that:
\begin{equation}\label{expconv}
\mathbf E\big(\eta_T^{(k+1)}\big) \leq   (D (1+\tilde CT^2e^{\tilde C T}))^k  \frac{T^k}{k!}.
\end{equation}
Combining \eqref{gmestimate},  \eqref {rho} and \eqref{expconv} we infer that for all $s \in [0,T]$
\begin{equation}\label{summuk}
\mathbf E[||\mu_s^{(k+1)} - \mu_s^{(k)}||] \mbox{ is summable over } k.
\end{equation}
Moreover, from Lemma \ref{Goetz} (applied with $\alpha=\zeta^{(k)}$, $\tilde\alpha=\zeta^{(k-1)}$, $\varphi=\mu^{(k)}$, $\tilde\varphi=\mu^{(k-1)}$), \eqref{SDEk}  and \eqref{summuk} we conclude that
 \begin{equation}\label{stabeta}
 \sup_{s\le T} \mathbf E [| (\zeta^{(k)}_{s}\vee \frac{1}{M})\wedge M - (\zeta^{(k-1)}_{s} \vee \frac{1}{M})\wedge M|] \mbox{ is summable over } k.
 \end{equation}
From \eqref{firstest}, Fubini, \eqref{expconv} and \eqref{stabeta},  we conclude that for all $s \in [0,T]$ and $i\in \mathbb N$
\begin{equation}\label{sumGk}
\mathbf P( G^{(k+1)}_s(i) \neq G^{(k)}_s(i))  \mbox{ is summable over } k.
\end{equation}
Hence for $\epsilon>0$ arbitrary and each finite subset $\mathcal M \subset [0,T] \times \mathbb N$, we have by Borel-Cantelli that there exists $k_0\in\N$ with
\begin{equation}\label{stabG} \mathbf P\big(G^{(k)}_s(i) = G^{(k_0)}_s(i) \mbox { for all } k\ge k_0 \mbox{ and } (s,i) \in \mathcal M\big) \ge 1-\epsilon.
\end{equation}
Now take $\ell$ as in Lemma \ref{colouringlem} and choose the (nonrandom) finite set  $\mathcal M \subset [0,T] \times \mathbb N$ so large and ``dense'' that with probability $1-\epsilon$ every fragment $\Gamma_\gamma$ whose root $\gamma$ is an element of $[0,T) \times \{1,\ldots, \ell\}$, contains an element of $\mathcal M$.   Lemma \ref{colouringlem} together with \eqref{defmk} and \eqref{stabG}  imply that
\begin{equation}\label{stabmu}
 \mathbf P(\sup_{s\le T} ||\mu_s^{(k)} - \mu_s^{(k_0)}|| \ge \epsilon) \le 2\epsilon  \mbox { for all } k\ge k_0.
\end{equation}
From \eqref{stabeta}, \eqref{stabmu}, \eqref{sumGk} and \eqref{alpha} and the choice of $\mathcal M$ we infer that $((G_s^{(k)}(i))_{i\in \mathbb N}, \mu^{(k)}_s, (\zeta_s^{(k)}\vee \frac{1}{M})\wedge M)$ converges uniformly in $s \in[0,T]$ as $k\to \infty$. In order to see that the limit satisfies \eqref{eqGCM} we recall that the Poisson point measures $\mathcal L$ and $\mathcal K$ do not change over the iterations, and note that the distribution of the mark $z$ which figures in \eqref{eqGCM} and \eqref{eqGk} is continuous, which provides the adequate continuity in the coefficient $q$ that is given by the update rule \eqref{fecdraw}. Theorem (6.4) of \cite{protter1977} shows that the limit also satisfies \eqref{eqzeta}.\\

\noindent \textbf{Step 4, Uniqueness:}
The argument from Lemma~\ref{lem:desc-mu} shows that for any solution of~\ \eqref{eqGCM}, \eqref{eqzeta}, $G_s$ admits type frequencies for all $s\geq 0$ a.s.
Uniqueness of the strong solution of \eqref{eqGCM}, \eqref{eqzeta} follows
by the same argument as in Step~3 where we now compare in~\eqref{backcomp} two solutions instead of two approximations (Note that this strategy was also successful in the simpler setting of~\cite{MR1681126}).
This concludes the proof of Proposition~\ref{propCM}. $\Box$
\\\\
For the {\em completion of the proof of Theorem \ref{Gprop}}
let us now relax the control of the total mass with the constant $M$. Fix again the constant $C\in (0,\infty)$ and consider the
following system of SDEs
\begin{align}
\begin{split}\label{eqGC}
G_s(j) &= G_0(j)+ \sum_{i=1}^{j-1}\int_{[0,s]}(G_{u-}(i)-G_{u-}(j))d\LL_{ij}(u)\\
 &+  \sum_{1\le i<k<j}\int_{[0,s]}(G_{u-}(j-1)-G_{u-}(j))d\LL_{ik}(u)
\\
&+\int_{[0,s]\times [0,C] \times [0,1]\times\{\beta,\delta\}} (q(G_{u-}(j), G_{u-},     \zeta_{u-}, z,w, \omega) - G_{u-}(j)) d\mathcal K_j(u,z,w,\omega),
\\
 \zeta_0&=v_0, \quad \quad
d\zeta_s = \zeta_sf(\zeta_s,  \mu^{G_s}\{A\}) ds  +  \zeta_s \, d\mathcal W_s, \qquad s\ge 0.
\end{split}
\end{align}
Proposition \ref{propCM} tells us that this system has for each $M$ a unique pathwise solution up to
 the stopping time $\sigma_M$ defined by \eqref{sigmaM}.
By projectivity, this shows that \eqref{eqGC} has a unique pathwise solution up to the time   at which its mass process $\zeta$ goes to extinction or explodes.
In view of \eqref{fecdraw}, the solution of  \eqref{eqG}, \eqref{totmassLD} stopped at the extinction time of $\zeta$ coincides with that of \eqref{eqGC} up to that time at which $\zeta$ exceeds $\sqrt {\frac C c}$ or $\frac C b$. Again by projectivity, this implies the assertion of  Theorem \ref{Gprop}. 
\hfill $\Box$


\section{From the neutral to the selective genealogy} \label{genealogy}
\label{sec_sel-gen}

For an initial configuration $(R_0,G_0)$ that is distributed according to the marked distance matrix distribution of a marked metric measure space, and for the independent stochastic input  $(\mathcal W, \mathcal L,   \mathcal K)$ specified in (I1), (I2), (I3) in Section~\ref{secmain}, Theorem \ref{Gprop} provides  an a.s.~unique solution $(\zeta,G) = (\zeta_s, G_s)_{s\in [0,\sigma)}$ of \eqref{totmassLD} and \eqref{eqG} up to time $\sigma$. From this lookdown representation we will construct in Sec.~\ref{LDgen}
 the process $(\zeta,X)= (\zeta_s, R_s, G_s)_{s\in[0,\sigma)}$ of type configurations and genealogical distance matrices, which will be turned in Sec.~\ref{symmgen} into the process $(\xi_t,Y_t)_{t\ge 0}$ of isomorphy classes of marked metric measure spaces that describe type distributions and sample genealogies. As will be proved in  Sec.~\ref{twomp}, the latter will provide the unique solution to the martingale problem formulated in Prop. \ref{martprobth}. We recall that we always assume that $(R_0,G_0)$ has the marked distance matrix distribution of a marked ultrametric measure space.

\subsection{The selective lookdown genealogy}\label{LDgen}
In this subsection we define the {\em selective lookdown space}. With regard to   \eqref{fecdraw} and \eqref{eqG}  we say that a point $(s,i)\in \mathbb [0,\sigma)  \times \mathbb N$ is {\em active} if
\begin{align*} \mathcal K_i \mbox{ has an atom in } (s,z,w,\beta)& \mbox{ for some } z \in \mathbb R_+, w \in [0,1]  \mbox{ such that }
\\  &z \le b \mu^{G_s}\{A\}\zeta_s
\end{align*}
or if
\begin{align*} \mathcal K_i \mbox{ has an atom in } &(s,z,w,\delta) \mbox{ for some } z \in \mathbb R_+, w \in [0,1]  \mbox{ such that }
\\  &z \le c \mu^{G_s}\{B\}(\zeta_s)^2 \quad \mbox{and} \quad G_{s-}(i) = A
, \mbox{ or }
\\  &z \le c \mu^{G_s}\{A\}(\zeta_s)^2 \quad \mbox{and} \quad G_{s-}(i) = B.
\end{align*}
In the first case we say that a fecundity event takes place at $(s,i)$, in the second case we say that a competition event happens at $(s,i)$.
Note that because $s\mapsto \zeta_s$ and  $s\mapsto \mu^{G_s}$ are a.s. continuous (by Theorem \ref{Gprop}), we can as well replace $s$ by $s-$ in the three inequalities.\\

Next we define the {\em selective ancestral lineage} of an element $\theta \in \bigcup_{\gamma\in\mathcal R} \Gamma_\gamma \subset Z$ (recall that $\Gamma_\gamma$ are the fragments of the neutral lookdown space $Z$, indexed by $\gamma\in\mathcal R$ as defined in Sec.~\ref{RaF}).
For this we trace the lineage of $\theta$ back into the past according to Remark \ref {neutallin}, until it hits an active point $(s,i) \in {\mathbb R}_+ \times \mathbb N$. If a competition event occurs at that point, then we continue the lineage at an element of $Z$ picked independently according to the neutral sampling measure $\mathfrak m_s$ defined by \eqref{neutsampmeas}. If a fecundity event happens at $(s,i)$, then we continue the lineage at an element of $Z$ picked independently according to  $\mathfrak m_s$ conditioned on the fragments of type $A$. The individuals on  the selective ancestral lineage of $\theta$ will be called the {\em selective ancestors} of $\theta$.\\

\begin{definition}\label{fromrhotoR}
	With regard to \eqref{inversetch}, which defines the mapping $s\mapsto t(s)$, we define the time-changed distance~$\rho$ as follows: For $(s_1,i_1)$ and $(s_2,i_2)$  in $[0,\sigma)\times\N$ that lie on the same selective ancestral lineage in the lookdown graph, we put
	\[\rho((s_1,i_1), (s_2,i_2)) := |t(s_1)-t(s_2)| = \big | \int_{s_1}^{s_2}  \zeta_u\, du\, \big |\, ;\]
	this is the time it takes from one point to the other when traveling along the selective ancestral lineage with speed $1/\zeta_u$ (cf. \eqref{tch}) at an intermediate point $(u,j)$. More generally we put
	\[\rho((s_1,i_1), (s_2,i_2)):= \rho((s_1,i_1), (s',j))+ \rho((s_2,i_2), (s',j))\]
	if the two selective ancestral lineages merge at some point $(s',j)$ with $s'\in [0,s]$; otherwise, if $(s_1,i_1), (s_2,i_2)$ have two distinct selective ancestors $(0,a_1)$ and $(0,a_2)$ at time $0$, we put
	\[\rho((s_1,i_1), (s_2,i_2)):= \rho((s_1,i_1), (0,a_1))+ \rho((s_2,i_2), (0,a_2))+R_0(a_1,a_2).\]
	For $s \in[0,\sigma)$ and $i_1,i_2 \in \mathbb N$ we set
	\begin{equation}\label{e:def-R}
		R_s(i_1,i_2):= \rho((s,i_1), (s,i_2)).
	\end{equation}
	In other words,
	\[R_s(i_1,i_2):= \begin{cases}
		2(t(s)-t(s')), &\mbox{ if the two selective ancestral lineages merge at  time } s' \in [0,s], \\
		2t(s) + R_0(a_1,a_2), &\mbox{ if the selective ancestors  } (0,a_1), (0,a_2) \mbox{  at time } 0 \mbox{ are different}.
	\end{cases}
	\]
We define $\hat Z$ as the completion of $([0,\sigma)\times\N)\cap \bigcup_\gamma\Gamma_\gamma$ with respect to $\rho$ and call $(\hat Z,\rho)$ the {\em selective lookdown space}.
\end{definition}
Because the set $\bigcup_{\gamma\in\mathcal R}\Gamma_\gamma$ is contained in $Z$, the selective lookdown space $(\hat Z,\rho)$ inherits the family of {\em sampling measures}  $\m_s$, $s>0$, from \eqref{neutsampmeas}, which remain probability measures by Lemma~\ref{lem:mass}.
As each fragment $\Gamma_\gamma$ is monotypic, we can also endow $\hat Z\times \I$ (equipped with the product topology) with the sampling measures $m_s$, defined by
\begin{equation}
\label{selm}
m_s(E\times E')=\sum_{\gamma=(u,i)\text{ with }G_u(i)\in E' }\m_s(\Gamma_\gamma \cap E),\quad  E'\subset\I, \, E\subset \hat Z \text{ measurable.}
\end{equation}
\begin{lemma}\label{l:5-1}
The measures $m_s$ defined in  \eqref{selm} obey
\begin{align}
\label{selappr}
m_s=\text{w-}\lim_{n\to\infty}\frac{1}{n}\sum_{i=1}^n\delta_{((s,i),G_s(i))}
\end{align}
for all $s\in(0,\sigma)$ on an event of probability $1$ that does not depend on $s$. Here, the weak limit in \eqref{selappr} can be understood either with respect to the metric $\rho^{(0)}$ or with respect to the metric $\rho$.
\end{lemma}
\begin{proof}
Let $T>0$, $\epsilon>0$. By Lemma~\ref{colouringlem}, there exists a.\,s.\ a random $\ell\in\N$ such that
\begin{equation*}
m_s\left(\bigcup_{\gamma=(u,i)\in\mathcal R: i\le\ell} \Gamma_\gamma \times\I\right)>1-\epsilon
\end{equation*}
for all $s\in[0,T\wedge \sigma)$.
We note that the restrictions of $\rho^{(0)}$ and $\rho$ to $([0,\sigma)\times\N)\cap\Gamma_\gamma$ induce the same topology and have the same completion, for each of the countably many $\gamma\in\mathcal R$. Hence,
by Corollary~\ref{corol:lemma1.8}, \eqref{selappr} holds on each $\Gamma_\gamma\times\I$.  This implies the assertion.
\end{proof}
In Definition \ref{fromrhotoR} we have defined the process $(R_s)$ of evolving genealogies in the lookdown setting in terms of the random metric $\rho$ and the total mass process $\zeta$. 

 \begin{corollary} \label{RGproper}
For $s\in(0,\sigma)$, the marked distance matrices $(R_s,G_s)$ and $(R_{s-},G_{s-})$ are proper on an event of probability $1$ that does not depend on $s$.
\end{corollary}
\begin{proof}
We denote by $\overline{\{s-\} \times \N}$ the closure of $\{s-\} \times \N$ in $(\hat Z,\rho)$.
We consider the isometry $\iota_s$ from $\overline{\{s-\}\times\N}\subset Z$ into $\mathbb T^{R_{s-}}$ that maps $(s-,i)$ to $i$, and we set $\hat\iota_s(\theta,h):=(\iota_s(\theta),h)$. Then $m^{R_{s-}, G_{s-}}=\hat\iota_s(m_{s-})$, and
\[\frac{1}{n}\sum_{i=1}^n\delta_{(i,G_{s-}(i))}=\hat\iota_s\left(\frac{1}{n}\sum_{i=1}^n\delta_{(s-,i),G_{s-}(i)}\right).\]
Lemma~\ref{l:5-1} now implies that $(R_{s-},G_{s-})$ is proper on an event that does not depend on $s$. The assertion on $(R_s,G_s)$ follows analogously.
\end{proof}

To make the genealogical distances $(R_s)_{s\in[0,\sigma)}$ constructed in~\eqref{e:def-R} and in~\eqref{eqRG} coincide a.\,s., we assume that the continuation of the ancestral lineage in this subsection is always done using the randomness from the mark $w$ in the corresponding competitive or selective event:
First let us consider the active points $(s,z,w,\delta)$ of $\mathcal K_i$, $i\in\mathbb N$, each of which corresponds to an active competitive death event at time $s$ and level $i$. Here we pick an individual according to $\mathfrak m_s$, or equivalently, an individual with type $(\theta,h)$ from $(Z,\rho^{(0)})\times\I$, according to $m^{(0)}_{s-}$. We assume that $(\theta,h)$ is realized as the image of the mark $w$ under a mapping that transports the uniform measure on $[0,1]$ into $m^{(0)}_{s-}$. To make the connection to~\eqref{eqRG}, we pick in addition an individual with type $(\theta',h')$  from $\hat Z\times\I$ according to $m_{s-}$. By Corollary~\ref{RGproper} and Lemma~\ref{lem:mass}, we may couple these picks such that $(\theta,h)=(\theta',h')$ (with equality between elements of the corresponding $\Gamma_\gamma$) for all points of $\mathcal K_i$ a.\,s.
Using the isometry $\iota_s:\hat Z\times \I\to \mathbb T^{R_{s-}}$ that maps each $(s-,j)$ to $j\in\mathbb T^{R_{s-}}$, and the isometry $\hat\iota_s(\tilde\theta,h):=(\iota_s(\tilde\theta),h)$ which satisfies $\hat\iota(m_{s-})=m^{R_{s-},G_{s-}}$, we can also couple such that $\iota(\theta',h')=\theta^{R_{s-},G_{s-}}$ for all such points of the $\mathcal K_i$,
with $(\theta^{R_{s-},G_{s-}},h^{R_{s-},G_{s-}})$ from~\eqref{updatedelta}.
Then the update of the genealogical distances satisfies $\vartheta_{i,\theta',h'}(R_{s-},G_{s-})=(R_{s},G_s)$ with $\vartheta_{i,\theta',h'}$ from~\eqref{selupdate2}.
For the points $(s,z,w,\beta)$ (corresponding to active selective birth events), we use the same argument but condition all the picks on $h=A$ resp.\ $h'=A$. This shows that the process $(R_s,G_s)$ with $R_s$ defined in~\eqref{e:def-R} (and $G_s$ from Theorem~\ref{Gprop}) is the pathwise solution of the SDE in~\eqref{eqRG} up to time $\sigma$.

Let us also prove that:
 \begin{proposition}\label{propexch}
 For each $s\in[0,\sigma)$, the pair $(R_s, G_s)$ is exchangeable conditionally given $(\zeta_u)_{u\le s}$.
\end{proposition}
\begin{proof} It suffices to work along the sequence of jump times of the restriction of the process $(R_s, G_s)_{s\ge 0}$ of marked distance matrices   to the first $n$ levels, where $n \in \mathbb N$ is arbitrarily fixed.
Between these jump times, the types of the individuals on the first $n$ levels remain unchanged, and the genealogical distance between each pair of such individuals grows deterministically with slope $2$. The jump times of this process are given by the Poisson processes $(\LL_{ij})_{1\le i<j\le n}$ and $(\mathcal K_i)_{1\le i\le n}$. Let us also recall that $(R_0, G_0)$ is exchangeable by the assumption that it has the marked distance matrix distribution of a marked ultrametric measure space.

(i) We first consider the jumps given by the neutral events, i.e. by  the  processes $(\LL_{ij})_{1\le i<j\le n}$.
We assume that the restriction of the process $(R_u, G_u)_{u\ge 0}$ to the first $n$ levels jumps at some time $s$ due to a neutral reproduction event.

Proceeding inductively we assume that the restriction of $(R_{s-}, G_{s-})$ to the first $n$ levels is exchangeable conditionally given $(\zeta_u)_{u\le s}$. That is, $((R(\ell,m))_{1\le \ell,m\le n}, (G(\ell))_{1\le \ell \le n}):= ((R_{s-}(\ell,m)_{1\le \ell,m\le n}, (G_{s-}(\ell))_{1\le \ell \le n})$ has the same distribution (conditionally given $(\zeta_u)_{u\le s}$) as
\begin{align}\label{deftildeR}
(\tilde R, \tilde G) := ((R(\sigma(\ell),\sigma(m)))_{1\le \ell,m\le n}, (G(\sigma(\ell)))_{1\le \ell \le n}),
\end{align}
 for each permutation $\sigma$ of $[n]$. Let also $(I,J)$ be independent and uniformly distributed on $\{(i,j): 1\le i < j \le n\}$, and let the random array $(\tilde R', \tilde G'):=\vartheta_{I,J}(\tilde R,\tilde G)$ be constructed from $(\tilde R, \tilde G)$ according to \eqref{updateg} and \eqref{updater} with $(i,j)=(I,J)$.

Putting
\begin{align*}
f_J(m):=\begin{cases} m &\mbox{ if } m < J,\\ m-1 &\mbox{ if } J<m \le n, \end{cases}
\end{align*}
we can write $\tilde R'$ as
\begin{align}
\tilde R'(\ell, m) = \begin{cases} 0 &\mbox{ if }  \ell, m \in \{I,J\}, \\
\tilde R(I,f_J(m)) &\mbox{ if } \ell   \in \{I,J\} \mbox{ and }  m \in [n] \setminus \{I,J\}, \\
\tilde R(f_J(\ell) ,f_J(m)) &\mbox{ if } \ell, m \in [n] \setminus \{I,J\}.
\end{cases}
\end{align}

Using \eqref{deftildeR} one checks readily that, for each permutation $\sigma$ of the numbers $1,\ldots,n$, the random array \\ $((\tilde R'(\sigma(\ell), \sigma(m)))_{1\le \ell, m\le n}, (\tilde G'(\sigma(\ell))_{1\le \ell \le n})$  has the same distribution as $(\tilde R'(\ell, m))_{1\le \ell, m\le n}, \tilde G'(\ell)_{1\le \ell \le n})$ conditionally given  $(\zeta_u)_{u\le s}$. As all the Poisson processes $\mathcal L_{i,j}$ have the same rate, the pair $(i,j)$ for which $\mathcal L_{i,j}$ has an atom at time $s$ is distributed as $(I,J)$. Hence, the above implies the desired exchangeability of  $(R_{s}, G_{s})$.

(ii) We now turn to the non-neutral events. Let $u$ be a time point at which one of the counting measures $\mathcal K_i(\cdot \times [0,C]\times [0,1] \times \{\beta, \delta\})$, $1\le i \le n$, has an atom  for which the corresponding activation condition on the r.h.s. of \eqref{fecdraw} is satisfied for
$v=\zeta_u$, $g= G_{u-}$ and $h=G_{u-}(i)$.
 Making use of part (i) and proceeding by induction, we assume that  the random array   $(R_{u-}^{(0)}(\ell,m), G_{u-}(\ell))_{1\le  \ell,m \le n}$ is exchangeable given $(\zeta_w^{(k-1)})_{w\le u}$. Let $i$ be that element of  $\{1,\ldots n\}$  for which $\mathcal K_i(\{u\}\times [0,C]\times [0,1] \times \{\beta, \delta\})=1$; because all the Poisson point measures $\mathcal K_\iota$ have the same intensity given $\zeta$, the level $i$ is uniformly chosen from $\{1,\ldots n\}$. According to the update rule \eqref{compcons}, \eqref{updatebeta}, conditionally given $(\zeta_w)_{w\le u}$, the exchangeability of the restriction of  $(R_{u-}, G_{u-})$ to the first $n$ levels  propagates to the exchangeability of the restriction of $(R_u, G_u)$ to the first $n$ levels.
\end{proof}

We remark that by using the sampling measures $\m_s$ for the independent picks needed to continue the ancestral lineage at competitive and selective events, we avoid the formalism of genetic markers which is used in Section 6 of \cite{MR1728556} to trace ancestral lineages.


The above construction and Theorem~\ref{Gprop} show that the process $(\zeta, X)$ is the pathwise unique solution the system of SDEs given by~\eqref{totmassLD}, \eqref{eqG} and~\eqref{eqRG}, which is driven by $(\mathcal L_{ij})$, $(\mathcal K_i)$, $\mathcal W$ and extends~\eqref{totmassLD}, \eqref{eqG} to include also the genealogical distances.
\subsection{Two well-posed martingale problems in the lookdown framework}\label{twomp}

Let $(\mathcal W, \mathcal L, \mathcal K)$ be as in Sec.\ref{sec:pathwiseconstructionLD}, choose some $M>0$ and let $(\zeta, R, G)$ be the unique strong solution of the system of SDE's~\eqref{totmassLD}, \eqref{eqG}, \eqref{eqRG}.  In this subsection we will prove Proposition \ref{MPA}, thus establishing a well-posed martingale problem for the (suitably stopped) process $(\zeta, R, G)$. For this, we follow the strategy outlined at the end of Sec.~\ref{section:LDrepresentation} (right after the statement of Proposition \ref{MPA}). We would like to enrich the process $(\zeta,R,G)$ by the addition of a component that keeps track of the number of events. A natural choice would be to add counting processes associated with $\mathcal{L}$ and $\mathcal{K}$, but in view of Kurtz' Markov Mapping Theorem (Corollary 3.5 in~\cite{Kurtz98}), we will rather add stationary components. The idea is to have a jump process that tracks the atoms of $\mathcal{L}_{ij}$ for every pair  $i<j$, for the natural births, and the atoms of $\mathcal{K}_k$ for all $k\in \mathbb{N}$, for the selective birth and potential death events. With $C_M$ defined in \eqref{CM1}, let $\mathscr V_M$ be a countable algebra of subsets of $[0,C_M] \times [0,1] \times \{\beta, \delta\}$ which generates the $\sigma$-algebra of Borel sets on $[0,C_M] \times [0,1] \times \{\beta, \delta\}$ and put
$$\mathscr H_M := \{\{(i,j)\}: 1\le i < j < \infty\} \cup \{\{k\} \times V: k \in \N, V \in \mathscr V_M\}. $$
We now define the  components $\Lambda^M_s(H)$, $H \in \mathscr H_M$, $s \ge 0$, of the additional   $\{-1,1\}^{\mathscr H_M}$-valued process~$\Lambda^M$. This process, together with $(\zeta, R, G)$, will constitute the solution of the well-posed (stopped) martingale problem specified in Proposition \ref{MPAlambda} below. For $s \ge 0$, let us define
\begin{align}\label{LambdaM} 
\begin{split}
 \Lambda^M_s(\{(i,j)\}) &:=  \Lambda^M_0(\{(i,j)\})\, (-1)^{\mathcal L_{ij}((0,s])}, \qquad 1\le i < j < \infty, \\ 
\Lambda^M_s(\{k\} \times V) &:= \Lambda^M_0(\{k\} \times V)\,(-1)^{\mathcal K_k([(0,s]\times V)},  \quad k\in \mathbb N,\,V \in \mathscr V_M,
\end{split}
\end{align}
where the random variables $\Lambda^M_0(H)$, $H \in \mathscr H_M$, are chosen as independent and uniformly distributed on $\{-1,+1\}$.
The process $\Lambda^M$ thus records the positions of the atoms of the Poisson point processes $\mathcal L$ and~$\mathcal K$; e.g. the  the process $\Lambda^M_s(\{k\} \times V)$  jumps at time $s>0$ if and only if  $\mathcal K_k$  has an atom in $\{s\} \times V$.\\

To prepare for a martingale problem for $(\zeta, R, G, \Lambda^M)$  stopped at $\sigma_M$, we define the state space
\begin{align}\label{EMlambda}
\widehat{E}_M := \left( \left(\tfrac 1M, M\right)\times \mathbb R^{\mathbb N^2}\times  \mathbb I^\mathbb N \times  \{-1,+1\}^{\mathscr H_M}\right) \cup \{\Delta_M\}
\end{align}
where  $ \left(\tfrac 1M, M\right)\times  \mathbb R^{\mathbb N^2} \times \mathbb I^\mathbb N \times \{-1,+1\}^{\mathscr C_M}$ is equipped with the product topology and a sequence $(v_n,r_n,g_n,\lambda_n)$ is said to converge to $\Delta_M$ if either $v_n \to \frac 1M$ or $v_n \to M$ as $n \to \infty$.
Note that the space $\widehat{E}_M$ is an extension of the space $E_M$ defined in \eqref{EMlambda1}.\\

Next we display the generator of $(\zeta, R, G, \Lambda^M)$ restricted to appropriate test functions $F = F(v,r,g, \lambda)$, where $v \in \mathbb R_+$, $r \in \mathbb R^{\mathbb N^2}$, $g \in \mathbb I^\mathbb N$ and $\lambda \in \{-1,+1\}^{\mathscr H_M}$. 

For $H \in \mathscr H_M$ we define $\varphi_H$ as the 
 projection from $\{-1,+1\}^{\mathscr H_M}$ to its $H$-component; thus we have the identity $\Lambda_s^M(H) = \varphi_H(\Lambda_s^M)$. 

With $D_{1,M}$ and $D_{2,M}$ as in Section \ref{sec:martingalepbzetaGLambda},
 let $D_{3,M}$ be the set of those functions $\varphi: \{-1,+1\}^{\mathscr H_M} \to \mathbb R$ which are of the form 
\begin{align}\label{def:varphi}
\varphi =  \prod_{(i,j)\in L}  \varphi_{\{i,j\}} \prod_{\{k\}\times V\in \mathscr K}\varphi_{\{k\}\times V}
\,
\end{align}
for some $n \in \mathbb N$, $L\subset \{(i,j):1\le i<j \le n\}$, and some finite subset $\mathscr K \subset  \{\{k\} \times V: k \in \N, V \in \mathscr V_M\}$ whose elements are disjoint subsets of $\{1,\ldots,n\}\times[0,C_M]\times[0,1]\times\{\beta,\delta\}$. We set 
\begin{equation}\label{def:K=UK}
K:= \bigcup_{H\in \mathscr K} H. 
\end{equation}
We will consider test functions of the form
\begin{align} \label{firstF}
F(v,r,g,\lambda) =f(v,r) \gamma(g) \varphi(\lambda),
\end{align}
for $(v,r,m,\lambda)\in\widehat E_M$ with $v\in(1/M,M)$, where $f \in D_{1,M}$, $\gamma \in D_2$ and $\varphi \in D_{3,M}$. Here we also assume that $F$ is continuous in $\Delta_M$. The smallest possible $n \in \mathbb N$ which fits to the required representations of $f$, $\gamma$ and $\varphi$ will be called the {\em degree} of~$F$. We write $F_{r(i,j)}$ for the partial derivative of $F$ with respect to the variable $r(i,j)$, and $F_{v}$ for partial derivative of $F$ with respect to $v$. \\

 Let $\vartheta_{i,j}$ and $\tilde \vartheta_{j,\theta, h'}$ be the updates acting on $(r,g)$ at the neutral and selective events, resepctively,  as defined in \eqref{updateg}, \eqref{updater},  \eqref{selupdate} and  \eqref{selupdate2}.
 For $(r,g) \in \mathbb R^{\mathbb N^2}\times  \mathbb I^\mathbb N$, we recall the definition of the  sampling measure ${\rm m}^{r,g}$ from Definition \ref{defproper}.  Let $\kappa'_{{\rm m}^{r,g}}$ be a measurable mapping defined on $[0,1]$ that transports the uniform distribution on $[0,1]$ into the  measure ${\rm m}^{r,g}$. Also, let $\kappa''_{{\rm m}^{r,g}}$ be a measurable mapping defined on $[0,1]$ that transports the uniform distribution on $[0,1]$ into the conditioned sampling measure ${\rm m}^{r,g}(d\theta,dh|h=A)$. We write $\mu^g$ for the second marginal of ${\rm m}^{r,g}$ and put $(\theta^{r,g}(w), h^{r,g}(w)):= \kappa'_{{\rm m}^{r,g}}(w)$, $(\tilde \theta^{r,g}(w), \tilde h^{r,g}(w)):= \kappa''_{{\rm m}^{r,g}}(w)$, $w \in [0,1]$. We now consider a function $F$ of degree $n$ as in \eqref{firstF}, with $L$, $\mathscr K$ the sets used in the definition of $\varphi$ \eqref{def:varphi}, and set $K$ as in \eqref{def:K=UK}. For all $v \in(1/M,M)$, all pairs $(r,g)\in  \mathbb R^{\mathbb N^2}\times  \mathbb I^\mathbb N$, all $\lambda \in  \{-1,+1\}^{\mathscr H_M}$,
 \begin{align}\begin{split}
\label{genzetaGlambda}
&\widehat{\mathbf A} F(v,r,g,\lambda)=\frac{v^2}{2}F_{vv}(v,r,g,\lambda)
+\big(bv^2\mu^g\{A\}-2cv^3\mu^g\{A\}\mu^g\{B\}\big)F_v(v,r,g,\lambda)\\
&+2v\sum_{1\leq i\neq j\leq n}F_{r(i,j)}(v,r,g,\lambda)\\
&+  \sum_{1\leq i<j\leq n}\big((-1)^{\mathbf 1_L(i,j)} F(v,\vartheta_{i,j}(r,g), \lambda)-F(v,r,g, \lambda)\big)\\
&+\sum_{j=1}^n \mathbf 1_{\{g(j)=B \} } \int_0^{cv^2\mu^g\{A\}} dz\, \int_0^1 dw  \,
\left((-1)^{\mathbf 1_K(j,z,w,\delta)}F(v,{\tilde\vartheta}_{j,\theta^{r,g}(w),h^{r,g}(w)}(r,g),\lambda)-F(v,r,g,\lambda)\right)\\
&+\sum_{j=1}^n \mathbf 1_{\{g(j)=B \} } \int_{cv^2\mu^g\{A\}}^{C_M} dz\, \int_0^1 dw  \,
\left((-1)^{\mathbf 1_K(j,z,w,\delta)}F(v,r,g,\lambda)-F(v,r,g,\lambda)\right)\\
&+\sum_{j=1}^n \mathbf 1_{\{g(j)=A \} } \int_0^{cv^2\mu^g\{B\}} dz\, \int_0^1 dw   \,
\left((-1)^{\mathbf 1_K(j,z,w,\delta)}F(v,{\tilde\vartheta}_{j,\theta^{r,g}(w),h^{r,g}(w)}(r,g),\lambda)-F(v,r,g,\lambda)\right)\\
&+\sum_{j=1}^n \mathbf 1_{\{g(j)=A \} } \int_{cv^2\mu^g\{B\}}^{C_M} dz\, \int_0^1 dw  \,
\left((-1)^{\mathbf 1_K(j,z,w,\delta)}F(v,r,g,\lambda)-F(v,r,g,\lambda)\right)\\
&+\sum_{j=1}^n\int_0^{bv\mu^g\{A\}}dz\, \int_0^1 dw\,  \left((-1)^{\mathbf 1_K(j,z,w,\beta)}F(v,{\tilde\vartheta}_{j,\tilde\theta^{r,g}(w),\tilde h^{r,g}(w)}(r,g),\lambda)-F(v,r,g,\lambda)\right)\\
&+\sum_{j=1}^n\int_{bv\mu^g\{A\}}^{C_M}dz\, \int_0^1 dw\,   \left((-1)^{\mathbf 1_K(j,z,w,\beta)}F(v,r,g,\lambda)-F(v,r,g,\lambda)\right),
\end{split}
\end{align}
and $\widehat{\mathbf A} F(\Delta_M)=0$. 
Let $\widehat{D}_M$ be the linear span of the constant real-valued functions on $\widehat{E}_M$ and all functions  of the form~\eqref{firstF}, and denote the extension of \eqref{genzetaGlambda} to $\widehat{D}_M$ again by $\widehat{\mathbf A}$.
\begin{proposition}\label{MPAlambda}
The process $(\zeta_{s\wedge \sigma_M}, R_{s\wedge \sigma_M},G_{s\wedge \sigma_M}, \Lambda^M_{s\wedge \sigma_M})_{s\ge 0}$ solves the  martingale problem $(\widehat{\mathbf A}, \widehat{D}_M)$, and this martingale problem  is well-posed.
\end{proposition}
\begin{proof}
a) For all $F \in \widehat{D}_M$,
\begin{align}\label{FMart}
F(\zeta_{s\wedge \sigma_M}, R_{s\wedge \sigma_M},G_{s\wedge \sigma_M},\Lambda^M_{s\wedge \sigma_M})-\int_0^{s\wedge \sigma_M} \widehat{\mathbf A} F(\zeta_u,R_u, G_u,\Lambda^M_u) \, du, \quad s\ge 0,
\end{align} is a martingale by It\^o's formula, since $(\zeta_{s\wedge \sigma_M}, R_{s\wedge \sigma_M},G_{s\wedge \sigma_M}, \Lambda^M_{s\wedge \sigma_M})_{s\ge 0}$ obeys up to time $\sigma_M$ the SDEs \eqref{totmassLD}, \eqref{eqG}, \eqref{eqRG} and the SDE for $\Lambda^M$ driven by $(\mathcal L_{ij})$ and $(\mathcal K_i)$.

b) Conversely, from any solution $(\widehat \zeta_{s\wedge \widehat \sigma_M},\widehat R_{s\wedge \widehat \sigma_M}, \widehat G_{s\wedge \widehat \sigma_M},\widehat \Lambda^M_{s\wedge \widehat \sigma_M})_{s\ge 0}$ to  the martingale problem $(\widehat{\mathbf A}, \widehat{D}_M)$ we can extract, up to the stopping time $\widehat \sigma_M$ (which is defined as in \eqref{sigmaM} but now for $\widehat \zeta$ instead of~$\zeta$) a Poisson point process $\widehat {\mathcal L}$ on $\mathbb R_+ \times \{(i,j): 1\le i < j<\infty\}$ and a Poisson point process  $\widehat {\mathcal K}$ on $\mathbb R_+ \times \mathbb N \times [0,C_M] \times [0,1] \times \{\beta,\delta\}$, such that $({\widehat L}, \widehat {\mathcal K})$ is equal in distribution to the corresponding restriction of $({\mathcal L},  {\mathcal K})$. We can then extract from  $(\widehat \zeta, \widehat G, \widehat {\mathcal L}, \widehat {\mathcal K})$ also a Brownian motion~$\widehat {\mathcal W}$ up to the stopping time $\widehat \sigma_M$. Taking $(\widehat {\mathcal W},  \widehat {\mathcal L}, \widehat {\mathcal K})$ as the source of randomness in the system given by~\eqref{totmassLD}, \eqref{eqG}, \eqref{eqRG} and in the definition of $\Lambda^M$, we infer from the pathwise uniqueness of that system that
\begin{align}\label{disteq}
(\widehat \zeta_{s\wedge \widehat \sigma_M},\widehat R_{s\wedge \widehat \sigma_M}, \widehat G_{s\wedge \widehat \sigma_M},\widehat \Lambda^M_{s\wedge \widehat \sigma_M})_{s\ge 0} \stackrel d= (\zeta_{s\wedge \sigma_M}, R_{s\wedge \sigma_M},G_{s\wedge \sigma_M},\Lambda^M_{s\wedge \sigma_M})_{s\ge 0},\end{align}
as asserted.
\end{proof}
Now we turn to the completion of the proof of Proposition \ref{MPA}, by establishing a well-posed martingale problem for $(\zeta,R,G)$. Recall from Sec. \ref{updateR} that the mapping $w \mapsto (\theta'(w), h'(w))$ which appears in \eqref{genzetaGlambda} is chosen such that, given $(R_{s-}, G_{s-})=(r,g)$, it transports the uniform distribution on $[0,1]$ into a pick from the sampling measure  $\mathrm  m^{r, g}$. Thus, for functions $F$ that are of the form \eqref{firstF} with $\varphi \equiv 1$ (and hence do not depend on $\lambda$) the operator $\widehat{\mathbf A}$ defined in \eqref{genzetaGlambda} turns into the operator~$\mathbf A$ defined in \eqref{genX}.
The proof of Proposition \ref{MPA} now follows from Proposition \ref{MPAlambda} together with Kurtz' Markov mapping theorem, see Corollary 3.5 in~\cite{Kurtz98}. The roles of the processes $X$ and $Y$ there are played by our processes $(\zeta, R, G, \Lambda^M)$ and $(\zeta, R, G)$, respectively.
In the initial distribution, the components of $\Lambda^M_0$ are i.i.d.\ Unif$\{-1,1\}$ distributed, and also for the kernel $\alpha$ appearing in Corollary~3.5 in \cite{Kurtz98}, we take $\alpha(v,r,g,\cdot)$ to be the i.i.d.\ Unif$\{-1,1\}$ distribution. The process $(\zeta,R,G)$ solves the martingale problem for $(\mathbf A, D_M)$ by~\eqref{eqRG}, and the martingale problem for $(\mathbf{\widehat A},\widehat D_M)$ is well-posed by Proposition~\ref{MPAlambda}.

To verify the assumption on the forward equation in~\cite[Corollary 3.5]{Kurtz98}, we apply Theorem~2.9 d) of~\cite{Kurtz98}: 
with $\widehat E_M$ specified in \eqref{EMlambda}, we consider the space $\widetilde E_M := \widehat E_M\times [0,1]^2  \times  (\R^{\N^2}\times\I^\N)^\N\times  (\R^{\N^2}\times\I^\N)^\N$, whose  topology  is obtained from the product topology by identifying all states with $\widehat E_M$-component $\Delta_M$ into a single state, also denoted by $\Delta_M$. For  for $F, v, r, g$ and $\lambda$ as in \eqref{genzetaGlambda} we define the {\em generator in the environment} $(y,w, (r_i,g_i)_i, (\tilde r_i, \tilde g_i)_i) \in [0,1]^2  \times  (\R^{\N^2}\times\I^\N)^\N \times  (\R^{\N^2}\times\I^\N)^\N$ as
\begin{align*}
		A^0F({\mathbf y}) = A^0 F (v,r,g,&\lambda,y,w,(r_i,g_i)_i,(\tilde r_i, \tilde g_i)_i)
		=\frac{v^2}{2}F_{vv}(v,r,g,\lambda)
		+\big(bv^2y-2cv^3y(1-y)\big)F_v(v,r,g,\lambda)\\
		&+2v\sum_{1\leq i\neq j\leq n}F_{r(i,j)}(v,r,g,\lambda)\\
		&+  \sum_{1\leq i<j\leq n}\big((-1)^{\mathbf 1_L(i,j)} F(v,\vartheta_{i,j}(r,g), \lambda)-F(v,r,g, \lambda)\big)\\
		&+\sum_{j=1}^n \mathbf 1_{\{g(j)=B \} } \int_0^{cv^2y} dz\,
		\left((-1)^{\mathbf 1_K(j,z,w,\delta)}F(v,r_j,g_j,\lambda)-F(v,r,g,\lambda)\right)\\
		&+\sum_{j=1}^n \mathbf 1_{\{g(j)=B \} } \int_{cv^2y}^{C_M} dz\,
		\left((-1)^{\mathbf 1_K(j,z,w,\delta)}F(v,r,g,\lambda)-F(v,r,g,\lambda)\right)\\
		&+\sum_{j=1}^n \mathbf 1_{\{g(j)=A \} } \int_0^{cv^2(1-y)} dz\, \,
		\left((-1)^{\mathbf 1_K(j,z,w,\delta)}F(v,r_j,g_j,\lambda)-F(v,r,g,\lambda)\right)\\
		&+\sum_{j=1}^n \mathbf 1_{\{g(j)=A \} } \int_{cv^2(1-y)}^{C_M} dz \,
		\left((-1)^{\mathbf 1_K(j,z,w,\delta)}F(v,r,g,\lambda)-F(v,r,g,\lambda)\right)\\
		&+\sum_{j=1}^n\int_0^{bvy}dz\,  \left((-1)^{\mathbf 1_K(j,z,w,\beta)}F(v,\tilde r_j, \tilde g_j,\lambda)-F(v,r,g,\lambda)\right)
		\\&+\sum_{j=1}^n\int_{bvy}^{C_M}dz\,  \left((-1)^{\mathbf 1_K(j,z,w,\beta)}F(v,r,g,\lambda)-F(v,r,g,\lambda)\right),
\end{align*}
and $A^0F(\Delta_M)=0$. We then define the  transition kernel  $\eta$ from  $\widehat E_M$ to $\widetilde E_M$ as follows:  For ${\mathbf x} := (v,r,g,\lambda)\in\widehat E_M$ with $v\in(1/M,M)$, we put
\[\eta({\mathbf x},\cdot)=\delta_{(v,r,g,\lambda,\mu^g\{A\})}\otimes dw\, \mathbf 1_{[0,1]}(w)\otimes
\bigotimes_{j\in\N}\delta_{{\tilde\vartheta}_{j,\theta^{r,g}(w),h^{r,g}(w)}(r,g)}\otimes \bigotimes_{j\in\N} \delta_{{\tilde\vartheta}_{j,\tilde\theta^{r,g}(w),\tilde h^{r,g}(w)}(r,g)},\]
and $\eta(\Delta_M,\cdot)=\delta_{\Delta_M}$. In words, the transition kernel $\eta$ records the input state $(v,r,g,\lambda)$, reads off from this the type frequency $y= \mu^g(A)$, chooses a uniformly distributed sampling seed $w$ and produces the output states of the type configuration and the distance matrix that arise by drawing from the sampling measure using $w$ and updating according to the rules \eqref{updatedelta} and \eqref{updatebeta}, for the case of competition as well as fecundity events.
We then have
\[\mathbf{\widehat A} F ({\mathbf x}) =\int \eta({\mathbf x},d{\mathbf y}) A^0 F({\mathbf y}).\]
We note that  $A^0 F$ is bounded and continuous for all $F \in \widehat D_M$. It follows that $\widehat D_M\times \{A^0F:F\in\widehat D_M\}$ is separable with respect to the $\ell_1$-product of the supremum metric, hence also $\widehat D_M\times \{\mathbf{\widehat{A}}F:F\in\widehat D_M\}$ has this property. That is, $\mathbf{\widehat A}$ satisfies Hypothesis 2.4 in~\cite{Kurtz98}. Also, $\widehat D_M$ is closed under multiplication and separates points. Moreover, for each $w,(r_i,g_i)_i,(\tilde r_i, \tilde g_i)_i$, the operator $F\mapsto A^0F(\cdot,w,(r_i,g_i)_i,(\tilde r_i, \tilde g_i)_i))$ is a pre-generator in the sense of~(2.1) of~\cite{Kurtz98}.
By Theorems 2.9 d) and 2.7 of \cite{Kurtz98}, it then follows that the assumption on the forward equation in \cite[ Corollary~3.5]{Kurtz98} is satisfied. This concludes the proof of Proposition \ref{MPA}.
\begin{remark}\label{KMMTH}
For our application of the Markov mapping theorem in the proof of Proposition  \ref{MPA} we resorted to \cite[Corollary~3.5]{Kurtz98} instead of the more recent variant \cite[Theorem A.2]{EtheridgeKurtz}, because we cannot guarantee the continuity of $\mathbf A F$  for $F\in {\widehat D}_M$, which is required in the latter. \cite[Corollary~3.5]{Kurtz98} requires a condition on the forward equation for an auxiliary operator $A^0$ which we were able to verify. A similar reasoning applies to the proof of Theorem \ref{martprobth} to which we turn now.
\end{remark}
\subsection{The symmetrized selective genealogy. Proof of Theorem \ref{martprobth}}\label{symmgen}
This section devoted to the proof of Theorem~\ref{martprobth}. We divide this proof into three steps. \\
\noindent \textbf{Step 1.}
Recall the definition of $D_M$ just after \eqref{genX}, and that of $\Phi_F$ in \eqref{eq:PhiF_APhiF}. The process $(\xi, Y)$ was defined in terms of $(\zeta, X) = (\zeta, R, G)$ by means of \eqref{def:psi}, using the time change \eqref{tch}. With regard to this time change we first claim that $(\zeta_{s\wedge \sigma_M},\psi( R_{s\wedge \sigma_M}, G_{s\wedge \sigma_M}))_{s\ge 0}$ solves the martingale problem $(\widetilde{\mathbb A}, \mathbb D_M)$, where the generator $\widetilde{\mathbb{A}}$ is defined by
\begin{equation}\label{defAtilde}
\widetilde{\mathbb A}\Phi_F (v,\chi)
:= \int \mathbf A F(v,r,g) \nu^\chi(dr, dg),\qquad F\in D_M,
\end{equation}
with the marked distance matrix distribution $\nu^\chi$  defined at the beginning of Section \ref{mppopsizeandgenealogy}, and where $\mathbb{D}_M$ is the domain of $\widetilde{\mathbb{A}}$ containing the linear span of all the functions $\Phi_F$ for $F\in D_M$.
We assume that the initial configuration $X_0=(R_0,G_0)$ is distributed according to the marked distance matrix distribution of $Y_0$.\\

The process  $(\zeta_{s\wedge \sigma_M},\psi( R_{s\wedge \sigma_M}, G_{s\wedge \sigma_M}))_{s\ge 0}$ arises from the process $(\zeta_{s\wedge \sigma_M}, R_{s\wedge \sigma_M},G_{s\wedge \sigma_M})_{s \ge 0}$  through the measurable mapping given by $(v,r,g) \mapsto (v,\psi(r,g))$, where $\psi(r,g)$  is the isomorphy class of the marked ultrametric measure space $(\mathbb{T}^r, r, \mathrm{m}^{r,g})$ and measurability of $\psi$ can be shown as in Section~10.2 of~\cite{Gu2}.

Due to Theorem~\ref{zetaRG}, the process  $(\zeta_{s\wedge \sigma_M}, R_{s\wedge \sigma_M},G_{s\wedge \sigma_M})_{s \ge 0}$ is Markovian.
Moreover, we know from Proposition~\ref{MPA} that the
process $(\zeta_{s\wedge \sigma_M}, R_{s\wedge \sigma_M},G_{s\wedge \sigma_M})_{s \ge 0}$ solves the martingale problem $(\mathbf A, D_M)$.
Hence, in order to prove our first claim,
it suffices to show that
\begin{multline}\label{e:A}
	\mathbf E\Big( \Phi_F (\zeta_{s_2\wedge \sigma_M}, \psi( R_{s_2\wedge \sigma_M}, G_{s_2\wedge \sigma_M})) 
	-  \Phi_F (\zeta_{s_1\wedge \sigma_M}, \psi( R_{s_1\wedge \sigma_M}, G_{s_1\wedge \sigma_M})) \\
	- \int_{s_1\wedge \sigma_M}^{s_2\wedge\sigma_M} \widetilde{\mathbb A} \Phi_F (\zeta_u,\psi( R_{u}, G_u))\, du\,\Big|\,
	 (\zeta_u, \psi( R_u, G_u))_{u\le s_1\wedge \sigma_M } \Big) =0\quad \text{a.\,s.}
\end{multline}
for all $F\in D_M$ and $0\le s_1\le s_2$.
The l.h.s.\ in~\eqref{e:A} equals a.\,s.\ the conditional expectation \\ $\E\big(\cdot \,\big|\,  (\zeta_u, \psi( R_u, G_u))_{u\le s_1\wedge \sigma_M } \big)$ of
\begin{multline}\label{e:AA}
	\mathbf E\Big( \Phi_F (\zeta_{s_2\wedge \sigma_M}, \psi( R_{s_2\wedge \sigma_M}, G_{s_2\wedge \sigma_M})) 
	-  \Phi_F (\zeta_{s_1\wedge \sigma_M}, \psi( R_{s_1\wedge \sigma_M}, G_{s_1\wedge \sigma_M})) \\
	- \int_{s_1\wedge \sigma_M}^{s_2\wedge\sigma_M} \widetilde{\mathbb A} \Phi_F (\zeta_u, R_{u}, G_u)\, du\,\Big|\,
	(\zeta_u, R_u, G_u)_{u\le s_1\wedge \sigma_M } \Big)
\end{multline}
by the tower property of conditional expectation.
By the Markov property of $(\zeta_{s\wedge \sigma_M}, R_{s\wedge\sigma_M}, G_{s\wedge\sigma_M})_{s\ge 0 }$ , the expression in~\eqref{e:AA} does a.\,s.\ not change when we replace $(\zeta_u,  R_u, G_u)_{u\le s_1\wedge \sigma_M }$ with $(\zeta_{s_1\wedge \sigma_M}, R_{s_1\wedge \sigma_M}, G_{s_1\wedge \sigma_M})$ in the conditioning. Moreover, we have
\begin{multline}\label{e:PhiF-F}
	\mathbf E \Big(\int
	F(\zeta_{s\wedge \sigma_M}, r,g) \nu^{\psi(R_{s\wedge \sigma_M},G_{s\wedge \sigma_M})}    (dr,dg)\,\Big|\, \zeta_{s\wedge\sigma_M},\zeta_{s_1\wedge \sigma_M}, R_{s_1\wedge \sigma_M}, G_{s_1\wedge \sigma_M}  \Big)\\
	= \mathbf E \Big( F(\zeta_{s\wedge \sigma_M},R_{s\wedge \sigma_M},G_{s\wedge \sigma_M}) \,\Big|\, \zeta_{s\wedge\sigma_M},\zeta_{s_1\wedge \sigma_M}, R_{s_1\wedge \sigma_M}, G_{s_1\wedge \sigma_M} \Big) \quad\text{a.\,s.}
\end{multline}
for all $s\ge s_1$, and for these $s$ we also have
\begin{multline}\label{e:phiAF-AF}
	\mathbf E \Big(\int
	\mathbf AF(\zeta_{s\wedge \sigma_M}, r,g) \nu^{\psi(R_{s\wedge \sigma_M},G_{s\wedge \sigma_M})}    (dr,dg) \,\Big|\, \zeta_{s\wedge\sigma_M},\zeta_{s_1\wedge \sigma_M}, R_{s_1\wedge \sigma_M}, G_{s_1\wedge \sigma_M}  \Big)  \\
	= \mathbf E \Big( \mathbf AF(\zeta_{s\wedge \sigma_M},R_{s\wedge \sigma_M},G_{s\wedge \sigma_M})
	\,\Big|\, \zeta_{s\wedge\sigma_M},\zeta_{s_1\wedge \sigma_M}, R_{s_1\wedge \sigma_M}, G_{s_1\wedge \sigma_M}  \Big) \quad\text{a.\,s.}
\end{multline}
by Fubini and an argument as in e.\,g.\ Proposition 10.3 of~\cite{Gu2} which uses that $(R_s,G_s)$ is exchangeable conditionally given $\zeta_s$ by Proposition~\ref{propexch}, and $(R_s,G_s)$ is a.\,s.\ proper by Corollary \ref{RGproper}. Using the definition of $\Phi_F$ and $\widetilde{\mathcal A}$, and~\eqref{e:PhiF-F}, \eqref{e:phiAF-AF}, we obtain that the expression in~\eqref{e:AA} equals a.\,s.\
\begin{multline*}
	\mathbf E\Big( F (\zeta_{s_2\wedge \sigma_M}, R_{s_2\wedge \sigma_M}, G_{s_2\wedge \sigma_M}) 
	-  F (\zeta_{s_1\wedge \sigma_M}, R_{s_1\wedge \sigma_M}, G_{s_1\wedge \sigma_M}) \\
	- \int_{s_1\wedge \sigma_M}^{s_2\wedge\sigma_M} \mathbf A F (\zeta_u, R_{u}, G_u)\, du\,\Big|\,
	(\zeta_u, R_u, G_u)_{u\le s_1\wedge \sigma_M } \Big)
\end{multline*}
which in turn is a.\,s.\ equal to $0$ by Proposition~\ref{MPA}. This shows the claim~\eqref{e:A}.\\

\noindent \textbf{Step 2.}
Next we show the well-posedness of the martingale problem $(\widetilde{\mathbb A}, \mathbb D_M)$.
From Proposition \ref{MPA} we recall that the martingale problem $(\mathbf A, D_M)$ is well-posed. Similar as in Sec.~\ref{twomp} we are going to apply Kurtz' Markov mapping theorem in the form of \cite[Corollary 3.5]{Kurtz98}.~Let us check the validity of the required assumptions using the notation from there. The state space of the ``coarse'' process $(\zeta_{s\wedge \sigma_M}, \psi( R_{s\wedge \sigma_M}, G_{s\wedge \sigma_M}))_{s\ge 0}$ is $E_0:= S_M$ defined in \eqref {DefSM}.
The state space of the ``fine'' process $(\zeta^{\sigma_M}, R^{\sigma_M},G^{\sigma_M})$ is  $E:=E_M$ defiend in \eqref{EMlambda1}.
As the mapping $\gamma$ from $E$ to $E_0$ and the probability kernel~$\alpha$ from $E_0$ to $E$ that both figure in \cite{Kurtz98} Corollary~3.5, we take
$$\gamma(v,r,g):= (v,\psi(r,g))\in E_0, \qquad \alpha((v,\chi),dv'\,dr\,dg):=\delta_v(dv')\otimes\nu^{\chi}(dr\,dg).$$
Then we have
$\gamma(\alpha(\chi,\cdot))=\delta_\chi$ by a reconstruction argument as in e.g.\ Proposition 10.5 of \cite{Gu2}. We can rewrite the operator $\widetilde{\mathbb A}$ defined in \eqref{defAtilde} as
\[\widetilde{\mathbb A}\Phi_F (v,\chi)
= \int \mathbf A F(v',r,g) \alpha((v,\chi),dv'\,dr\,dg),\qquad F\in D_M.\]
In view of Step 1 and the well-posedness of the martingale problem $(\mathbf A, D_M)$, we can now infer the well-posedness of the martingale problem $(\widetilde{\mathbb A}, \mathbb D_M)$ as well as the Markov property of its solution from \cite{Kurtz98} Corollary~3.5.

To verify the assumption on the forward equation in~\cite[Corollary 3.5]{Kurtz98}, we proceed in analogy to the proof of Proposition \ref{MPA}: again we apply Theorem~2.9 d) of \cite{Kurtz98}, now  defining $\eta$ as a transition kernel from $E_M$ to $E_M\times [0,1]^2  \times  (\R^{\N^2}\times\I^\N)^\N\times  (\R^{\N^2}\times\I^\N)^\N$ where the topology on the latter space is obtained from the product topology by identifying all states with $E_M$-component $\Delta_M$ into a single state, also denoted by~$\Delta_M$. For $(v,r,g,\lambda)\in E_M$ with $v\in(1/M,M)$, we define $\eta((v,r,g),\cdot)$ by
\[\eta((v,r,g),\cdot)=\delta_{(v,r,g,\mu^g\{A\})}\otimes dw\mathbf{1}_{w\in[0,1]}\otimes
\bigotimes_{j\in\N}\delta_{{\tilde\vartheta}_{j,\theta^{r,g}(w),h^{r,g}(w)}(r,g)}\otimes \bigotimes_{j\in\N} \delta_{{\tilde\vartheta}_{j,\tilde\theta^{r,g}(w),\tilde h^{r,g}(w)}(r,g)},\]
and $\eta(\Delta_M,\cdot)=\delta_{\Delta_M}$. The operator $A^0$ figuring in \cite[Theorem~2.9]{Kurtz98} is defined in complete analogy to the operator $A^0$ in the proof of Proposition \ref{MPAlambda}, now without the component $\lambda$ in the argument. This results in
\[\mathbf{A} F ({\mathbf x}) =\int \eta({\mathbf x},d{\mathbf y}) A^0 F({\mathbf y}),\]
as requred in \cite[Theorem~2.9]{Kurtz98}.
The validity of the assumption on the forward equation now follows as in the end of Sec. \ref{twomp}.

\noindent \textbf{Step 3.} Steps 1 and 2 together show that $(\zeta_{s\wedge \sigma_M},\psi( R_{s\wedge \sigma_M}, G_{s\wedge \sigma_M}))_{s\ge 0}$ is the unique  solution of the martingale problem $(\widetilde{\mathbb A}, \mathbb D_M)$ and is Markovian. The assertion of Theorem \ref{martprobth}   now follows from the time-change relation $\widetilde{\mathbb A} = v\mathbb A$ (see e.\,g.\ Theorem~6.1.3 of~\cite{ethierkurtz}). This completes the proof of Theorem \ref{martprobth}.

\subsection{Proof of Propositions \ref{propcont}, \ref{prop:projection} and \ref{strongsol_type_process}}\label{proof_of_two_props}
\begin{proof}[Proof of Proposition~\ref{propcont}]  The process $Y$ takes its values in the space $\mathbb M$ of marked genealogies that is equipped with the marked Gromov-weak topology, see Sec \ref{SecGen}. According to~\cite{DGP11}, this topology is metrized by the so-called marked Gromov-Prohorov metric, and~\cite{DGP11} Definition  3.1 ensures that the Gromov-Prohorov distance of two elements $\chi, \chi' \in \mathbb M$ is bounded from above by the Prohorov distance of $m$ and $m'$, where the marked ultrametic measure spaces $(\tau,d,m)$ and $(\tau,d,m')$ are  representatives of the isomorphy classes $\chi$ and $\chi'$ in a common embedding. In our situation the common embedding of the representatives of $Y_{t_1}$ and $Y_{t_2}$ happens in the selective lookdown space   $(\hat Z,\rho)$, and the two measures in the embedding are the sampling measures $m_{s(t_1)}$ and $m_{s(t_2)}$, with  $m_s$ defined in \eqref{selm}. Since for each $M\in \mathbb N$ the time change $t \mapsto s(t)$ given by \eqref{tch} is bi-continuous up to the stopping time $\tau_M$, it suffices to show that a.s. the map $s\mapsto m_s$ is continuous in the weak topology on $(\hat Z,\rho)$. This latter continuity, however, is   a consequence of Lemmas~\ref{lem:cont-frag-masses} and~\ref{colouringlem} and the fact that the fragments $\Gamma_\gamma$ are monotypic.
\end{proof}
\begin{proof}[Proof of Proposition~\ref{prop:projection}.] We proceed in two steps. In the whole proof, let $M>0$ be fixed and let us consider all the processes stopped at $\tau_M$. For the sake of notation, we omit here the stopping times $\tau_M$.\\
First, we explain how to obtain the martingale problems for $\xi^A$ and $\xi^B$ and second, we compute the brackets of the corresponding martingales. \\

\noindent \textbf{Step 1:} For any $f:\mathbb R_+\times\mathbb I\to\mathbb R$ of class $\Co^\infty$ that is supported in $(1/M,M)$ with respect to its first component and bounded, we can associate a function $F=f\circ \gamma_1$ of degree 1 on $\R_+\times \R^{\N^2} \times \I^\N$. Such a function $F$ belongs to $D_M$ with $F(v,r,g)=f(v,g(1))$ for all $(v,r,g)\in\mathbb R_+\times\mathbb R^{\mathbb N^2}\times\mathbb I^{\mathbb N}$, and our purpose is to rewrite the martingale problem \eqref{martprobY} for such test function $F$. \\

For the first term in the left hand side of \eqref{martprobY}, we have that:
\begin{align}\label{etape7}
\Phi_F(\xi_t,Y_t)=\int f\big(\xi_t, g(1)\big)  \nu^{Y_t}(dr,dg)= f\big(\xi_t, A\big) \mu_t\{A\}+ f\big(\xi_t, B\big) (1-\mu_t\{A\}),
\end{align}
since under the marked distance matrix distribution $\nu^{Y_t}$, the type configuration corresponds to an i.i.d. sequence drawn from $\mu_t\{A\} \delta_A (dh)+ (1-\mu_t\{A\})\delta_B(dh)$.\\

Now, let us compute the second term of the left hand side of \eqref{martprobY}. For our choice of function $F$, we have from \eqref{genX} that:
\begin{multline*}
\mathbb{A}\Phi_F(v,Y_t)\\
\begin{aligned}= & \int \nu^{Y_t}(dr,dg)\Big\{\frac{v}{2}f_{vv}(v,g(1))
+\big(bv\mu^g\{A\}-2cv^2\mu^g\{A\}\mu^g\{B\}\big)f_v(v,g(1))\notag\\
&+cv
\int {\mathrm m}^{r,g}(d\theta,dh) \Big[\ind_{\{g(1)=B \} }\mu^g\{A\}
(f(v,h)-f(v,B)) + \ind_{\{g(1)=A \} }\mu^g\{B\}
(f(v,h)-f(v,A))\Big]\notag\\
&+b\int {\mathrm m}^{r,g}(d\theta,dh)  \ind_{\{h=A\}} (f(v,A)-f(v,g(1)))\Big\}\\
= & \int \nu^{Y_t}(dr,dg)\Big\{\frac{v}{2}f_{vv}(v,g(1))
+\big(bv\mu^g\{A\}-2cv^2\mu^g\{A\}\mu^g\{B\}\big)f_v(v,g(1))\notag\\
&+cv
\Big(\ind_{\{g(1)=B \} }(\mu^g\{A\})^2
 - \ind_{\{g(1)=A \} }(\mu^g\{B\})^2
\Big)(f(v,A)-f(v,B))\\
+ & b\mu^g\{A\}  (f(v,A)-f(v,g(1)))\Big\},
\end{aligned}
\end{multline*}by recalling that the projection of ${\mathrm m}^{r,g}(d\theta,dh)$ on its second component gives the type frequencies $\mu^g\{A\}$ and $\mu^g\{B\}$. Under $\nu^{Y_t}(dr,dg)$, $\mu^g\{A\}$ (resp. $\mu^g\{B\}=1-\mu^g\{A\}$) is constant and equal to $\mu_t\{A\}$ (resp. $\mu_t\{B\}$) and $g(1)$ is a random variable that takes the values $A$ and $B$ with probabilities $\mu_t\{A\}$ and $1-\mu_t\{A\}$. We then deduce that:
\begin{align}
\mathbb{A}\Phi_F(v,Y_t)= & \mu_t\{A\} \frac{v}{2}f_{vv}(v,A)+(1-\mu_t\{A\}) \frac{v}{2}f_{vv}(v,B) \nonumber\\
+ & \big(bv\mu_t\{A\}-2cv^2 \mu_t\{A\} (1-\mu_t\{A\})\big)\big( \mu_t\{A\} f_v(v,A)+(1-\mu_t\{A\})f_v(v,B) \big)\nonumber\\
+ &cv \mu_t\{A\} (1-\mu_t\{A\}) (2\mu_t\{A\}-1) \big(f(v,A)-f(v,B)\big)\nonumber\\
+ & b\mu_t\{A\}  (1-\mu_t\{A\}) (f(v,A)-f(v,B)).\label{etape6}
\end{align}Choosing $f$ in \eqref{etape7} and \eqref{etape6} such that $f(v,h)=v \ind_{\{h=A\}}$ and replacing $v$ by $\xi_t$, we find that:
\begin{align}
M^A_t:=  \xi^A_t -\xi^A_0  - \int_0^{t} \Big(b \xi^A_u - c \xi^A_u \xi^B_u \Big) du, \qquad t\geq 0,
\label{etape9}
\end{align}is a local martingale (when stopped at $\tau_M$, $M^A_{.\wedge \tau_M}$ is a square integrable martingale) started at 0. This local martingale is also continuous by Theorem~\ref{Gprop}. We can proceed similarly to find that $M^B_t:=\xi^B_t -\xi^B_0 + \int_0^t c \xi^A_u \xi^B_u du$ is also a continuous local martingale started at 0.\\

\noindent \textbf{Step 2:} Let us now compute the brackets $\langle M^A\rangle_.$, $\langle M^B\rangle_.$ and $\langle M^A,M^B\rangle_.$. Proceed similarly as in Step 1 for functions $f:\mathbb R_+\times\mathbb I^2\to\mathbb R$ of class $\Co^\infty$ and supported in $(1/M,M)$ with respect to its first component, and to which we associate $F=f\circ \gamma_2$ of degree 2 on $\R_+\times \R^{\N^2} \times \I^\N$. Such a function $F$ belongs to $D_M$ with $F(v,r,g)=f(v,g(1),g(2))$ for all $(v,r,g)\in\mathbb R_+\times\mathbb R^{\mathbb N^2}\times\mathbb I^{\mathbb N}$.\\

For the choice of $f(v,h_1,h_2)=v^2 \ind_{\{h_1=A\}}\ind_{\{h_2=A\}}$, we obtain that:
\begin{align}
(\xi^A_t)^2 -\int_0^t \Big(\xi^A_u+2b(\xi^A_u)^2-2c (\xi^A_u)^2 \xi^B_u\Big) du,\qquad t\geq 0,
\end{align}
is a continuous local martingale. Using Itô's formula on \eqref{etape9}, we also have that:
\begin{align}
(\xi^A_t)^2  - \int_0^t \Big(2b(\xi^A_u)^2-2c (\xi^A_u)^2 \xi^B_u\Big) du - \langle M^A\rangle_t, \qquad  t\geq 0,
\end{align}
is a continuous local martingale. From the comparison of these two expressions, we deduce that:
\begin{equation}
\langle M^A\rangle_t = \int_0^t \xi^A_u du, \qquad  t\geq 0.
\end{equation}
In a similar way, the choices of $f(v,h_1,h_2)= v^2 \ind_{\{h_1=B\}}\ind_{\{h_2=B\}}$ and $f(v,h_1,h_2)= v^2 \ind_{\{h_1=A\}}\ind_{\{h_2=B\}}$ allow us to compute $\langle M^B\rangle_.$ and $\langle M^A,M^B\rangle_.$. Using Levy's representation theorem \cite[Th. IV.3.6, p.141]{revuzyor}, we deduce that there exists on an enlarged probability space two independent Brownian motions $W^A$ and $W^B$ such that $dM^A_t=\sqrt{ \xi^A_t}dW^A_t$ and $dM^B_t=\sqrt{\xi^B_t}dW^B_t$.
\end{proof}
\begin{proof}[Proof of Proposition~\ref{strongsol_type_process}.]
Let us consider the following stopping times, for any $\varepsilon > 0$:
\begin{equation}\label{eq:tauA-tauB}
\tau^A_\varepsilon=\inf\{t \geq 0,\ \xi^A_t\leq \varepsilon\},\qquad \mbox{ and }\qquad \tau^B_\varepsilon=\inf\{t \geq 0,\ \xi^B_t\leq \varepsilon\},
\end{equation}with the usual convention that $\inf \emptyset =+\infty$. Before $\tau^A_\varepsilon\wedge \tau^B_\varepsilon$, the diffusion coefficients are Lipschitz continuous (with a Lipschitz constant of order $1/\sqrt{\varepsilon}$). Classical results (e.g. \cite[Ch. IV]{ikedawatanabe}) ensure strong uniqueness of the stopped processes $(\xi^A_{.\wedge \tau^A_\varepsilon \wedge \tau^B_\varepsilon},\xi^B_{.\wedge \tau^A_\varepsilon \wedge \tau^B_\varepsilon})$ for all $\varepsilon>0$. Let $\tau_0 = \inf\{ t \geq 0, \xi_t^A=0$ or $\xi_t^B=0 \}$. By the continuity of the processes, $\lim_{\epsilon \downarrow 0} \tau_\epsilon^A \wedge \tau_\epsilon^B = \tau_0$. Once one of the processes $\xi^A$ or $\xi^B$ has touched zero, it remains trapped there and the other process coincides with a standard (possibly drifted) Feller diffusion $\bar{\xi}^A$ or $\bar{\xi}^B$:
\begin{align}
\label{twotype-independent}
& d\bar{\xi}_t^A= b \bar{\xi}_t^A dt + \sqrt{\bar{\xi}_t^A} dW_t^{A} \quad \mbox{ and }\quad \bar{\xi}_t^B=0\\
\mbox{ or }&
\bar{\xi}_t^A=0\quad \mbox{ and }\quad d\bar{\xi}_t^B= \sqrt{\bar{\xi}_t^B} dW_t^{B}.\nonumber
\end{align}
The latter diffusions are well studied (see e.g. \cite[Ch. IV.8]{ikedawatanabe}) and we have strong existence and uniqueness for \eqref{twotype-independent}, and consequently also for \eqref{twotype}.

In order to prove the asserted long time behavior of $(\xi^A, \xi^B)$ we observe that the two (independent) Feller diffusions $\bar{\xi}^A$ and $\bar{\xi}^B$ appearing in  \eqref{twotype-independent} also provide dominating processes for $(\xi^A,\xi^B)$ (see e.g. \cite[Th.  VI.1.1]{ikedawatanabe}). The process $\bar{\xi}^A$  remains nonnegative for all $t\ge 0$, and $0$ is a trap. In the following, we assume $c>0$. Denoting $\bar{\tau}^A_0=\inf\{t \geq 0,\ \bar{\xi}^A_t=0\}$, it is known that $\P(\bar{\tau}^A_0<+\infty)\in (0,1)$ and that on the set $\{\bar{\tau}^A_0=+\infty\}$, $\lim_{t\rightarrow +\infty} \bar{\xi}^A_t=+\infty$ a.s. (see \cite[Corollary~2, p.190]{bertoinlevyprocesses}).
The  process $\bar{\xi}^B$  gets extinct almost surely in finite time: $\P(\bar{\tau}^B_0<+\infty)=1$ with $\bar{\tau}^B_0=\inf\{t \geq 0,\ \bar{\xi}^B_t=0\}$. The process $(\bar{\xi}^A,\bar{\xi}^B)$ dominates stochastically $(\xi^A,\xi^B)$. As a consequence, $\xi^B$ gets extinct in finite time almost surely. Overall, either $\xi^A$ touches zero before $\xi^B$ and the whole process then goes to extinction, or $\xi^A_{\tau^B_0}>0$ and there is a positive probability that $\tau^A_0=+\infty$ and when this happens, $\lim_{t\rightarrow +\infty}\xi_t^A=+\infty$ a.s.
\end{proof}

%
%

\section{Outlook: An extension to multiple types and mutations}\label{sec:multitype}
The previous sections were restricted to a prototype example with two types and without mutation. Indeed, we believe that this example, which allowed for a trade-off between conciseness and elaboration, is best suited for displaying our novel {\em pathwise} approach to the joint evolution of population size, type configuration and genealogy.

In this concluding section we give a brief outlook to a more general situation, without going into further details.
Let now the type space $\mathbb  I$ be a compact group. Again we write  $\mu_t $ for the relative type frequencies and $\xi_t$ for the total mass of the population at time $t$, and we put $\Xi_t := \xi_t\mu_t$. The state space of $(\mu_t)$  is $ M^1(\mathbb I)$, the set of probability measures on $\mathbb I$, and that of $(\Xi_t)$ is    $M(\mathbb I)$, the set of finite measures on $\mathbb I$, equipped with the weak topology.

 Let $\underline b =  \underline b(h)$ and $\underline  c=\underline  c(h,h')$ be bounded, measurable mappings from $\mathbb I$ to $\mathbb R_+$ and  $\mathbb I \times \mathbb I$ to $\mathbb R_+$, respectively. For $\rho  \in M^1(\mathbb I)$ and $h\in \mathbb I$ we put $c(h,\rho):= \int_{\mathbb I} \underline c(h,h') \rho(dh')$. Finally, let $h\mapsto \ell(h,\cdot)$ be a measurable map from $\mathbb I$ to $M^1(\mathbb I)$.

We say that $\Xi$ is an interactive Dawson-Watanabe process with fecundity function $b$, competition kernel~$c$ and mutation kernel $\ell$ if for all $f \in \mathcal C(\mathbb I)$, the continuous functions on $\mathbb I$,
\[\int_{\mathbb I} f(h)\, \Xi_t(dh) - \int_0^t \int_{\mathbb I}\left(f(h)(b(h)-c(h,\mu_u)\xi_u)+\int_{\mathbb I}  (f(h')-f(h))\ell(h,dh')\right) \Xi_u(dh) \,du\]
is a continuous martingale with quadratic variation
\[\int_0^t  \int_{\mathbb I}f^2(h)\, \Xi_u(dh) \,du,\]
cf. \cite{MR1681126}, Example 4.6 for the non-interactive case. Putting $f\equiv 1$ we see that the total mass process $\xi_t := \Xi_t(\mathbb I)$, $t\ge 0$, is required to be a weak solution of the SDE
\begin{equation}\label{xigeneral}
d\xi_t  = \left(\xi_t\int_{\mathbb I}b(h)\mu_t(dh)\, - \xi_t^2 \int_{\mathbb I}c(h,\mu_t)\, \mu_t(dh)\right)\,dt+\sqrt{\xi_t}\, dW_t.
\end{equation}
Our prototype example fits into this framework with $\mathbb I=\{A,B\}$, $\ell \equiv 0$ and   $\underline b(A)  = b$, $\underline b(B) = 0$, $\underline c(A,A)=\underline c(B,B)=0$, $\underline c(A,B)=\underline c(B,A)=c$.

In order to arrive at an analogue of Theorem \ref{Gprop} in this more general framework one has to modify the update rule \eqref{fecdraw}. In addition to the symbols $\beta$ and $\delta$ that indicate ``birth" or ``death'' as the $4^{th}$ component  of the Poisson point measures $\mathcal K_i$, we now have a third symbol $\lambda$ that figures for ``mutation'': the $\mathcal K_i$ are now a family of independent Poisson  processes on $\mathbb R_+\times \mathbb R_+ \times [0,1]\times \{\beta,\delta,\lambda\}$, with $\mathcal K_i( \cdot \times \{\beta\})$, $\mathcal K_i( \cdot \times \{\delta\})$ and $\mathcal K_i( \cdot \times \{\lambda\})$ having Lebesgue intensity measure.

With the abbreviation $\mu^{g,b} (dh') :=  \frac {b(h')\mu^g(dh')}{\int b\, d\mu^g}$,
the update rule \eqref{fecdraw} is  modified to
\begin{align*}
q(h, g,v, z, w, \beta) &:= \begin{cases}\kappa\big(\mu^{g,b},w) &\phantom{AAA}\mbox{ if } z \le \int  b\, d\mu^g\, \,  v,  \\ h  &\phantom{AAA} \mbox{ otherwise,} \end{cases}
\\
q(h, g,v, z, w, \delta) &:= \begin{cases}\kappa(\mu^g, w) &\phantom{AAAA}\mbox{ if } z \le c(h,\mu^g)\,  v^2 ,  \\ h  &\phantom{AAAA}\mbox{ otherwise.} \end{cases}
\\
q(h, g,v, z, w, \lambda) &:= \begin{cases}\kappa(\ell(h,\cdot),w)&\phantom{AA} \mbox{ if }z\leq v,  \\ h  &\phantom{AA}\mbox{ otherwise.}\end{cases}
\end{align*}
Likewise, the total mass process $(\zeta_s)$ in the lookdown timescale, which is another ingredient of Theorem~\ref{Gprop}, will be a time change of  $(\xi_t)$ under  \eqref{tch}, turning  \eqref {xigeneral} into
\[d\zeta_s =  \left(\zeta_s\int_{\mathbb I}b(h)\mu^{G_s}(dh) -\zeta_s^2\int_{\mathbb I} c(h,\mu^{G_s})\mu^{G_s}(dh)\right)\, \zeta_s\, ds \,+ \,\zeta_s \, d\mathcal W_s.\]

\noindent \textbf{Acknowledgements:} We thank G\"otz Kersting for helping us with the proof of Lemma \ref{Goetz}. A.B. was supported as a postdoc by CONACyT in an earlier phase of this project. She gratefully acknowledges the kind hospitality of Goethe-University Frankfurt. S.G. has been supported in part at the Technion by a Zeff Fellowship, a Minerva fellowship of the Minerva Gesellschaft fuer die Forschung mbH (11/17-10/19), and by Israel Science Foundation (ISF) grant No. 1382/17, Binational Science Foundation (BSF) award 2018330. S.G. has also been supported by Deutsche Forschungsgemeinschaft (DFG, German Research Foundation) research grant contract number 2337/1-1, project 432176920. S.K.~has been funded by the DFG - Project number 393092071. V.C.T. has been funded by Labex CEMPI (ANR-11-LABX-0007), Labex B\'ezout (ANR-10-LABX-58), the Chaire "Mod\'elisation Math\'ematique et Biodiversit\'e" of Veolia Environnement-Ecole Polytechnique-Museum National d'Histoire Naturelle-Fondation X and the European Union (ERC-AdG SINGER-101054787).  A.W. received partial support through DFG project   \mbox{WA 967/4-2} in the SPP 1590. Also, A.B., S.K., V.C.T. and A.W. would like to thank the Institute for Mathematical Sciences for supporting their visit to the IMS, National University of Singapore in 2017, where progress on this project was made. We also thank two anonymous referees for a careful reading and valuable hints that helped to improve the presentation.

\bigskip

{\footnotesize
\bibliographystyle{abbrv}

\providecommand{\noopsort}[1]{}\providecommand{\noopsort}[1]{}\providecommand{\noopsort}[1]{}\providecommand{\noopsort}[1]{}

}

\end{document}